\documentclass[10pt,reqno]{amsart}
\allowdisplaybreaks
\usepackage{amsmath,amstext,amssymb,epsfig}
\usepackage{graphicx}
\usepackage{mathrsfs}
\usepackage{xcolor}
\calclayout
\makeatletter
\renewcommand{\@seccntformat}[1]{\bf\csname the#1\endcsname.}
\renewcommand{\section}{\@startsection{section}{1}
	\z@{.7\linespacing\@plus\linespacing}{.5\linespacing}
	{\normalfont\upshape\bfseries\centering}}
\renewcommand{\@biblabel}[1]{\@ifnotempty{#1}{#1.}}
\makeatother
\theoremstyle{plain}
\newtheorem{thm}{Theorem}[section]
\newtheorem{lem}[thm]{Lemma}
\newtheorem{prop}[thm]{Proposition}
\newtheorem{cor}[thm]{Corollary}

\theoremstyle{definition}
\newtheorem{ex}[thm]{Example}
\newtheorem{defn}[thm]{Definition}
\newtheorem{rem}{Remark}[section]

\usepackage{tikz-cd}
\usepackage[parfill]{parskip}
\usepackage[german]{varioref}
\usepackage[all]{xy}
\usepackage{cancel}

\def\T{{\mathcal T}}
\def\>{\succ}
\def\<{\prec}

\def\b{\beta}
\def\a{\alpha}
\def\l{\lambda}
\def\p{\partial}

\def\m{\mu}

\begin{document}
	\title[Sania Asif\textsuperscript{1}, Wang Yao\textsuperscript{2}]{BiHom-(pre-)Poisson conformal algebra}
	%%%%%%%%%%%%%%%%%%%%%%%%%%%%%%%%%%%%%%%%%%%%%%%%%%%%%%%%%%%%%%%%%%%%%%%%%%%%%%%%
	\author{Sania Asif\textsuperscript{1}, Wang Yao\textsuperscript{2}}
   \address{\textsuperscript{1}School of Mathematics and statistics, Nanjing University, Nanjing, Jiangsu Province, PR China.}
	\address{\textsuperscript{2}School of Mathematics and statistics, Nanjing University, Nanjing, Jiangsu Province, PR China.} 
		\email{\textsuperscript{1}11835037@zju.edu.cn}
		\email{\textsuperscript{2}wangyao@nuist.edu.cn}

	%%%%%%%%%%%%%%%%%%%%%%%%%%%%%%%%%%%%%%%
	\keywords{BiHom-associative conformal algebra, BiHom-Lie conformal algebra, BiHom-Poisson conformal algebra, BiHom-preLie conformal algebra, BiHom-pre-Poisson conformal algebra, conformal representation}
	\subjclass[2000]{Primary 17B65, 17B10,15A99, Secondary 16G30}
	%%%%%%%%%%%%%%%%%%%%%%%%%%%%%%%%%%%%%%%%%%%%%%%%%%%%%%%%%%%%%%%%%%%%%%%%%%%%%
	%%%%%%%%%%%%%%%%%%%%%%%%%%%%%%%%%%%%%%%%%%%%%%%%%%%%%%%%%%%%%%%%%%%%%%%%%%
	\date{\today}
	\thanks{This work is supported by the Jiangsu Natural Science Foundation Project (Natural Science Foundation of Jiangsu Province), Relative Gorenstein cotorsion Homology Theory and Its Applications (No.BK20181406).}
%%%%%%%%%%%%%%%%%%%%%%%%%%%%%%%%%%%%
	\begin{abstract}
		The aim of this study is to introduce the notion of BiHom-Poisson conformal algebra, BiHom-pre-Poisson conformal algebra, and their related structures. We show that we can construct many new BiHom-Poisson conformal algebras for a given BiHom-Poisson conformal algebra. Moreover, the tensor product of two BiHom-Poisson conformal algebras is also a BiHom-Poisson conformal algebra. We further describe the conformal bimodule and representation theory of BiHom-Poisson conformal algebra. In addition, we define BiHom-pre-Poisson conformal algebra as the combination of BiHom-preLie conformal algebra and BiHom-dendriform conformal algebra under some compatibility conditions. We also demonstrate that how to construct BiHom-Poisson conformal algebra from BiHom-pre-Poisson conformal algebra and provide the representation theory for BiHom-pre-Poisson conformal algebra. Finally, a detailed description of $\mathcal{O}$-operators and Rota-Baxter operators on BiHom-Poisson conformal algebra is provided. \end{abstract}
	\footnote{second author is the corresponding author.}
	\maketitle \section{Introduction}The conformal algebra is the growing mathematical structure yielding variety of important structures such as associative conformal algebra, (pre-)Lie conformal algebra, (tri-)dendriform conformal algebra, (Zinbiel)-Leibniz conformal algebra and (pre-)Poisson conformal algebras that are important in the development of mathematics and physics. This algebra was first arise in the study of vertex algebra in \cite{CK} where the conformal modules of infinite dimensional Lie algebras were discussed. It shows that conformal algebra is basically related to infinite dimensional algebra satisfying locality property. Later on, classification of irreducible modules over Virasoro conformal algebra was discussed by Wu and Yuan in \cite{WY}. Structure, derivations, representations and cohomology theory of above mentioned conformal algebras is discussed in \cite{ALHW, GZW2, H, LZ}. Recently, conformal algebras with respect to operator algebras such as Rota-Baxter operators and Nijenhuish operators are also been studied where the structure theory of conformal operator algebras and their cohomology theory is discussed (see \cite{AWW,Y}).
	\par The research on an algebra in the presence of structural map, called a twist map has been extensively explored in \cite{AW, LMS, LS} and termed as Hom-type algebra. This category comprising various algebras including Hom-associative algebra, Hom-Lie algebra etc. Makhlouf and its collaborators explored Hom-Lie algebra, Hom-associative algebra, Hom-dialgebras (see \cite{AM}). Whereas, Frieger and Gohr in \cite{FG}, provided the sufficient condition for the Hom-associative algebra to be associative in detail. However, Armakan and Razavi in \cite{AR}, provided the sufficient condition for the Hom-Lie algebra's context. This concept is studied with reference to conformal algebra in \cite{AWW, AYW, GDW}. Later on, it was discovered that an algebra in the presence of two structural maps is called a BiHom-type algebra, which has also been widely studied  in \cite{CQ, GMMP, GZW}. In these papers, the authors discussed the structure theory of  algebra, BiHom-Lie algebra and and BiHom-bialgebra. Study of BiHom-bialgebra includes the study of BiHom-algebra and BiHom-coalgebra.
	\par On the other hand, the Poisson algebra is combination of associative algebra and Lie algebra in the presence of Leibniz identity. This is very important mathematical structure that generalize three structures and have been widely studied by various researchers. Kosmann-Schwarzbach explored how to get Gerstenhaber algebras from the poisson algebras in \cite{KS}. Later on, Poisson algebra in BiHom setting was studied in \cite{AL, L, ML}. As an associative algebra yield dendriform algebra, where a Lie-algebra yields a pre-Lie algebra, in the same way combination of dendriform algebra and pre-Lie algebra in the presence of compatibile conditions give rise to a pre-Poisson algebra. This structure is also very important and studied by various researcher in \cite{A}. However, pre-Poisson algebra is also studied in Hom and BiHom-setting in \cite{ LI, LS}. Recently, the study of (pre-)Poisson conformal algebra and Hom-(pre-)Poisson conformal algebra have been conducted in \cite{CHM, KP, LZ, Ls} and many useful results have explored, i.e., in \cite{CHM}, Chtioui and his collaborators studied the quantization of pre-(Hom)Poisson conformal algebra, followed by the quantization of Poisson conformal algebra in \cite{LZ} by Liu and Zhou. However, in \cite{KP}, Kolensikov studied the universal enveloping of Poisson conformal algebra. Getting motivation from the above cited literature, we observed that despite the  extensive research in this field, still  there are many unanswered questions and research gaps. Therefore, structural theory of BiHom-Poisson conformal algebra and its related structures are deemed essential. This paper aims to address these gaps by presenting many important structure concerning BiHom-Poisson conformal algebras. It introduces the notion of BiHom-Poisson conformal algebra and BiHom-pre-Poisson conformal algebra and  constructs new BiHom-Poisson conformal algebra and BiHom-(pre-) Poisson conformal algebra from the existing ones, and describes the representation theory of BiHom-(pre-)Poisson conformal algebra. Moreover, it highlights the relationship of BiHom-(pre-)Poisson conformal algebra with $\mathcal{O}$-operator and Rota-Baxter operator.
	\\ Collectively, Our main goal is to show that how to move from BiHom-Poisson conformal algbera to the BiHom-pre-Poisson conformal algbera and vice versa.  \\
	\begin{enumerate}
		\item  First case can be illustrated by the following diagram:\\
		\begin{tikzcd}[column sep=large, row sep=large]
			\parbox{10em}{\centering \textit{BiHom-dendriform\\ conformal algebra}} + \parbox{10em}{\centering \textit{BiHom-preLie\\ conformal algebra}} \arrow[r, "", align=center] \arrow[d, swap, "{\begin{tabular}{@{}c@{}}$x \<_\l y + x\>_\l y$,\\ $x*_\l y- \a^{-1}\b(y) *_{-\p-\l} \a\b^{-1}(x)$\end{tabular}}"] & \parbox{10em}{\centering \textit{BiHom-pre-Poisson conformal algebra}} \arrow[d, "{\begin{tabular}{@{}c@{}}$x \<_\l y + x\>_\l y$,\\ $x*_\l y- \a^{-1}\b(y) *_{-\p-\l} \a\b^{-1}(x)$\end{tabular}}"] \\
			\parbox{10em}{\centering \textit{BiHom-associative\\ conformal algebra}} + \parbox{10em}{\centering \textit{BiHom-Lie\\ conformal algebra}} \arrow[r, swap, "", align=center] & \parbox{10em}{\centering \textit{BiHom-Poisson\\ conformal algebra}}
		\end{tikzcd}\\
		\item Second case/ reverse case can be illustrated in the following diagram:\\
		\begin{tikzcd}[column sep=large, row sep=large]
			\parbox{10em}{\centering \textit{BiHom-associative\\ conformal algebra}} + \parbox{10em}{\centering \textit{BiHom-Lie\\ conformal algebra}} \arrow{r}{} \arrow{d}[swap]{R(x)\cdot  _\l y, x\cdot_\l R(y), [R(x)_\l y]}&  \parbox{10em}{\centering \textit{BiHom-Poisson conformal algebra}} \arrow{d}{xR(x)\cdot_\l y, x\cdot_\l R(y), [R(x)_\l y]} \\
			\parbox{10em}{\centering \textit{BiHom-dendriform\\ conformal algebra}} + \parbox{10em}{\centering \textit{BiHom-preLie\\ conformal algebra}}\arrow{r}[swap]{} &  \parbox{10em}{\centering \textit{BiHom-pre-Poisson\\ conformal algebra}}
		\end{tikzcd}
	\end{enumerate}
	 \par Apart from the present section, this paper is divided into four additional  sections. In Section $2$, we provide some preliminary definitions and notions about  and BiHom- Lie conformal algebras to support our findings in the next sections. In Section $3$, we define BiHom-Poisson conformal algebra and constructed new structures out of it. We further show that tensor product of two BiHom-Poisson conformal algebra is a BiHom- Poisson conformal algebra. Further, we introduce the representation theory for the BiHom-Poisson conformal algebra. Moving on to Section $4$, we introduce the BiHom-pre-Poisson conformal algebra followed by BiHom-pre Lie conformal algebra and provide it's representation theory. Finally, in Section $5$, we show the relationship of algebras under consideration with the $\mathcal{O}$-operator and Rota-Baxter operator. 
		\section{Preliminaries}
	In this section we recall important notions about BiHom-associative conformal algebra and BiHom-Lie conformal algebra to define the BiHom-Poisson conformal algebra 
	\begin{defn}A BiHom-conformal algebra $B$ is a $\mathbb{C}[\p]$module, equipped with the $\mathbb{C}$- bilinear map $\cdot_\l:B\otimes B \to B[\l]$ in such a way that $x\otimes y\mapsto x_\l y$  and two structural maps $\a,\b:B\to B$ satisfying the following identities:
	\begin{equation*}
	(\p x) \cdot_\l y = -\l (x\cdot_\l y) ,\quad x\cdot_\l (\p y)=(\p+\l)( x\cdot_\l y) ,\quad \a\p=\p\a,\quad  \b\p=\p\b,\end{equation*} for all $x,y\in B$ and $\l,\m\in \mathbb{C}$.
	We denote it by $(B, \a,\b, \cdot_\l)$. 	
	\end{defn}
	\begin{defn}
		Let $(B, \a,\b, \cdot_\l)$ be a BiHom-conformal algebra, it is said to be BiHom-associative conformal algebra if the following associative condition holds for $ x,y,z\in B$ and $\l,\m \in \mathbb{C}$.
		\begin{equation*}
		\a(x)_\l(y_\m z)=(x_\l y)_{\l+\m}\b(z).
		\end{equation*}\end{defn}
	A BiHom-associative conformal algebra is said to be multiplicative if the following condition hold 
\begin{equation*}
\a(x_\l y)=\a(x)_\l \a(y),\quad \b(x_\l y)=\b(x)_\l \b(y)
\end{equation*}
	Moreover, a BiHom-associative conformal algebra is said to be regular, if the twist maps $\a,\b$ are invertible. A BiHom-associative conformal algebra is called commutative if \begin{equation*}
	x\cdot_\l y= y\cdot_{\p-\l}x.
	\end{equation*}
	\begin{ex}
	Consider that $(B,*_\l)$ is an associative conformal algebra, Let's define two structure maps $\a,\b : B \to B$  on the associative conformal algebra, such that $ \a(x*_\l y)=\a(x)*_\l \a(y) , \b(x*_\l y)=\b(x)*_\l \b(y), $ for all $x, y\in B$ and $\l,\m\in \mathbb{C}$. Then $(B,\a(.*_\l.), \b(.*_\l.),\a,\b)$ yield a multiplicative BiHom-associative conformal algebra. \end{ex} Define the $m$-th product on a BiHom-associative conformal algebra $ B $ for $m\in \mathbb{N}$ by $a_\l b=\sum_{m=1}^{\infty} \l^m (x_{(m)}y)$ in such a way that $\a\p=\p\a$, $\b\p=\p\b$.  Note that, we can translate all the BiHom-associative conformal algebra' conditions in terms of $m-th$ products as follows:\begin{enumerate}
	\item  $(\p x)_{(m)}y= -m(x_{(m-1)}y)$,\item $(x)_{(m)}(\p y) = \p(x_{(m-1)}y) + m(x_{(m-1)} y)$ and\item $ \sum_{i=0}^{n} {\binom{n}{i}}(x_{(i)}y)_{(m+n- i)}\b(z) - \a(x)_{(n)}(y_{(m)}z) = 0. $
\end{enumerate}
	 Same as BiHom-associative conformal algebra, we define BiHom-Lie conformal algebra as follows:
	\begin{defn}
		A BiHom-Lie conformal algebra $B$ is a $\mathbb{C}[\p]$-module equipped with the $\mathbb{C}$-bilinear map $[\cdot_\l\cdot]:B\otimes B \to B[\l]$ in such a way that $x \otimes y\mapsto [x_\l y]$, and two structural maps $\a,\b$ satisfying the following identities:
		\begin{eqnarray} &\a\p=\p\a,~~ \b\p=\p\b \\&
		[(\p x) _\l y] = -\l[ x_\l y] ,~ [x_\l (\p y)]=(\p+\l) [x_\l y] ,\\&[\b(x)_\l \a (y)] = - [\b (y)_{-\p-\l} \a(x)] ,\\&
		[\b^2(x)_\l [\b (y)_\m \a (z)]] + [\b^2 (y)_\m [\b (z)_{-\p-\l}\a(x)]] + [\b^2(z)_{-\p-\l-\m} [\b(x)_\l \a (y)]] = 0.
	\\&\nonumber or \\ &\nonumber[\a\b(x)_\l[y_\m z]] = [[\b(x)_\l y]_{\l+\m}\b(z)] + [\b(y)_\m[\a(x)_\l z]].
		\end{eqnarray}for all $x,y,z\in B$ and $\l,\m\in \mathbb{C}$. Note that $3rd$ identity is called as  conformal BiHom-skew-symmetry while the $4th$ one is called as conformal BiHom-Jacobi identity.
	\end{defn} Similar to the BiHom-assocative conformal algebra, we can define the  $mth$-product of a BiHom-Lie conformal algebra.
	A finite BiHom-Lie conformal algebra is considered finitely generated as a $\mathbb{C}[\p]$-module. The rank of a BiHom-Lie conformal algebra $B$ is refers to its rank as a $\mathbb{C}[\p]$-module. A BiHom-Lie conformal algebra $(B, [._{\l} .], \a,\b)$ is called as regular if twist maps $\a$ and $\b$ are both invertible. 
	\begin{ex}	Let $(B,[\cdot_\l\cdot])$ be an Lie conformal algebra,  Lets define two structure maps $\a,\b$  on $B$, such that \begin{equation*}  \a([x_\l y])=[\a(x)_\l \a(y)] ,\quad  \b([x_\l y])=[\b(x)_\l \b(y)], \textit { for all }   x, y\in B,\quad \l,\m\in \mathbb{C}.
		\end{equation*} Then the tuple $(B,\a([._\l.]), \b([._\l.]), \a, \b)$ yield a multiplicative BiHom-Lie conformal algebra.
	\end{ex} Let $(B,\a,\b, *_\l)$ be a BiHom-associative conformal algebra then $\l$-bracket , defined by $$[x_\l y]=x*_\l y-y*_{-\p-\l} x$$ yields us a BiHom-Lie conformal algebra, denoted by $(B,\a,\b, [\cdot_\l\cdot])$ .
	\begin{defn}\label{def2.6}Let $(B, *_\l, \a, \b)$ be a  conformal algebra, and let $(M, \phi, \psi)$ be a $\mathbb{C}[\p]$- module equipped with the $\mathbb{C}$- linear maps $\phi, \psi: M\to M$ satisfying  $\phi\p=\p\phi$, $\psi\p=\p\psi$.  Let $l, r : B\to  gc(M)$, be  two $\mathbb{C}-$ bilinear maps. The quadruple $(M,l,r,\phi,\psi)$ is called a conformal representation of $B$, if th following equations hold
		\begin{eqnarray}
	 l(\p x)_\l m&= -\l(l(x)_\l  m),\\
		l(x)_\l(\p m) &= (\p+\l)(l(x)_\l m),\\
		r(\p x)_\l m &= -\l(r(x)_\l m),\\
		r(x)_\l(\p m) &= (\p+\l)(r(x)_\l m),\\
		\phi(l(x)_\l m) &= l(\a(x))_\l\phi(m),\\
		\phi(r(x)_\l m) &= r(\a(x))_\l \phi(m),\\
		\psi(l(x)_\l m) &= l(\b(x))_\l \psi(m),\\
			\psi(r(x)_\l m) &= r(\b(x))_\l \psi(m),\\ 
		l(x *_\l y)_{\l+\m}\psi(m) &= l(\a(x))_\l (l(y)_\m m), \\
		r(x *_\l y)_{\l+\m}\phi(m) &= r(\b(y))_\l (r(x)_\m m), \\
l(\a(x))_\l (r(y)_\m m) &= r(\b(y))_\m (l(x)_\l m), 
		\end{eqnarray}
		for all $x, y \in B$, $m \in M $ and $\l,\m\in \mathbb{C}$. Note that  $l(x)_\l m = x *_\l m$ and $ r(x)_\l m = m*_{-\p-\l} x. $ 
		
	\end{defn}
	\begin{prop}
		Let $(l , \phi, \psi, M)$ be a conformal-representation of a  conformal algebra $(B, *_\l, \a, \b)$, where $M$ is a $ \mathbb{C}[\p] $- module, $\phi,\psi$ are  $ \mathbb{C} $-linear maps, satisfying $\p\phi=\phi\p$ and $\p\psi=\psi\p$. Then, the direct sum $B\oplus M$ of vector spaces is turned into a  conformal algebra by defining the twist maps $\a+\phi, \b+\psi$ and $\l$-multiplication $*'_\l$ in $B \oplus M$ as follows
		\begin{equation}
		\begin{aligned}(x_1 +m_1)*'_\l(x_2 + m_2) &:= x_1 *_\l x_2 + (l(x_1)_\l{ m_2} + r(x_2)_{-\p-\l} m_1),\\
		(\a \oplus \phi)(x  + m) &:= \a(x ) + \phi(m),\\
		(\b \oplus \psi)(x  + m) &:= \b(x) + \psi(m ),
		\end{aligned}
		\end{equation}for all $x,x_1, x_2 \in B$, $m, m_1, m_2 \in M$ and $\l,\m\in \mathbb{C}$.
	\end{prop}
	We denote this  conformal algebra by $(B\oplus M, *'_\l, \a+\phi, \b+\psi)$, or simply $(B\ltimes_{l,r,\a,\b,\phi,\psi}M)$. 
	
	\begin{defn}\label{defrepBLA}A conformal representation of a BiHom-Lie conformal algebra $(B, [\cdot_\l\cdot ], \a,\b)$ on a $\mathbb{C}[\p]$-module $M$ equipped with a $\mathbb{C}$-linear maps $\phi,\psi: M \to M$ satisfying $\phi\p=\p\phi$, $\psi\p=\p\psi$ is a $ \mathbb{C}$-linear map $ \rho : B \to gc(M) $, such that the following equalities are satisfied:\begin{equation*}\begin{aligned}\rho(\p x)_\l m &= -\l(\rho(x)_\l  m), \\
		\rho(x)_\l(\p m) &= (\p+\l)(\rho(x)_\l m),\\
		\phi(\rho(x)_\l m) &= \rho(\a(x))_\l \phi(m),\\
		\psi(\rho(x)_\l m) &= \rho(\b(x))_\l \psi(m),\\
		\rho([\b(x)_\l y])_{\l+\m}\psi(m) &= \rho(\a\b(x))_\l \rho(y)_{\m}m -\rho(\b(y))_{\m} (\rho(\a(x))_{\l}m).\end{aligned}\end{equation*}for all $x,y\in B$, $m \in M$ and $\l,\m\in \mathbb{C}$.\end{defn}The conformal representation of the BiHom-Lie conformal algebra is denoted by $(\rho , \phi, \psi, M)$.
	\begin{prop}Let $(\rho , \phi, \psi, M)$ be a representation of a BiHom-Lie conformal algebra $(B, [\cdot_\l\cdot], \a, \b)$, where $M$ is a $ \mathbb{C}[\p] $-module, $\phi,\psi$ are  $ \mathbb{C} $-linear maps, satisfying $\p\phi=\phi\p$ and $\p\psi=\psi\p$. Moreover, assume that $\a,\psi $ are bijective maps. Then, the direct sum $B\oplus M$ of vector spaces is turned into a BiHom-Lie conformal algebra by  the $\l$-bracket $[\cdot_\l\cdot]_\rho$ in $B \oplus M$ defined as follows
		\begin{equation}
		\begin{aligned}[(x_1 +m_1)_\l(x_2 + m_2)]_{\rho} &:= [{x_1} _\l x_2] + ({\rho(x_1)}_\l m_2 - {\rho(\a^{-1}\b (x_2))}_{-\p-\l}\phi\psi^{-1}( m_1)),\\
	(\a \oplus \phi)(x  + m) &:= \a(x ) + \phi(m),\\
	(\b \oplus \psi)(x  + m) &:= \b(x) + \psi(m ),
		\end{aligned}
		\end{equation}for all $x,x_1, x_2 \in B,m, m_1, m_2 \in M$.
	\end{prop}We denote this BiHom-Lie conformal algebra by $(B\oplus M, [\cdot_\l\cdot]_\rho, \a+\phi, \b+\psi)$, or simply $(B\ltimes_{\rho,\a,\b,\phi,\psi}M)$. 
	\begin{ex}Consider that $ad:B\to gc(B)$ is a linear map defined by $ ad(x)_\l y=[x_\l y] $ for all $x, y\in B $. Then $(B,\a\,\b, ad)$ is the conformal representation of th BiHom-Lie conformal algebra. It is also known as adjoint representation of $B$.
	\end{ex} Now, we can define the BiHom analog of Poisson conformal algebra in the following section.
\section{BiHom-Poisson conformal algebra and its representations}
\begin{defn}\label{defBPA}A BiHom-Poisson conformal algebra $(B,*_\l,\a,\b,[\cdot_\l\cdot])$ is a tuple consisting of $\mathbb{C}[\p]$-module $B$, two $\mathbb{C}$-bilinear maps $*_\l, [\cdot_\l\cdot]: B\otimes B \to B$, two linear maps $\a,\b : B \to B$, such that the following identities hold:
		\begin{enumerate}
			\item $(B,*_\l, \a, \b)$ is a commutative  conformal algebra.
			\item $(B, [\cdot_\l\cdot], \a, \b)$ is a BiHom-Lie conformal algebra.
			\item The BiHom-Leibniz identity
		\begin{equation}\label{liebniz}
		[\a\b(x)_\l ( y*_\m z) ] = ([\b(x)_\l y])*_{\l+\m} \b(z) + \b(y)*_\m ([\a(x)_\l z]),
		\end{equation} for all $ x,y,z \in B$ and $\l,\m \in \mathbb{C}$.
		\end{enumerate}
		In a BiHom-Poisson algebra $(B, [\cdot_\l\cdot], *_{\l}, \a, \b)$, the operations $*_\l$ and $[\cdot_\l\cdot]$ are called the  conformal  product and the BiHom-Poisson conformal bracket, respectively.\end{defn}When $\a=\b$, we get Hom-Poisson conformal algebra, however, for $\a=\b=id$ , we simply get a Poisson conformal algebra.
	
	\begin{defn}
		Let $(B, [._\l.], *_\l, \a,\b)$ be a BiHom-Poisson conformal algebra. If $Z(B) = \{x \in B| [x_\l y]= x\cdot_\l y = 0, \forall x,y \in B\}$, then $Z(B)$ is called the centralizer of $B$.
	\end{defn}
\begin{rem}Let $(B, [\cdot_\l\cdot],*_\l, \a,\b)$ a BiHom-Poisson conformal algebra, then \begin{enumerate}
			\item $(B, [\cdot_\l\cdot],*_\l, \a,\b)$ is multiplicative if
			\begin{equation*}
			\begin{aligned} 
			&\a[x_\l y]= [\a(x)_\l\a(y)],~~\b[x_\l y]= [\b(x)_\l\b(y)]
		~~and ~~\\
			&\a(x*_\l y)=\a(x)*_\l \a(y) ,~~ \b(x*_\l y)=\b(x)*_\l \b(y).
			\end{aligned}
			\end{equation*}
			\item $(B, [\cdot_\l\cdot],*_\l, \a,\b)$  is said to be regular if $\a$ and $ \b $ are bijective maps.
			\item  $(B, [\cdot_\l\cdot],*_\l , \a,\b)$  is said to be involutive if $\a$ and $\b$ satisfiy $\a^2=id=\b^2.$
			\item Let $(B', [\cdot_\l\cdot]',*'_\l, \a',\b')$  be another BiHom-Poisson conformal algebra. A morphism $f: B\to B'$ is a linear map such that following conditions are satisfied:
		$$ f[x_\l y] = [f(x)_\l f(y)]' ~~
			and 
			~~ f (x * _\l y)  = f(x)*'_\l f(y) ~~
			and  \a f=f\a'~~, ~~\b f= f\b' $$
		\end{enumerate}
	\end{rem}
	\begin{prop} 
		 Let $(B, *_\l, \a,\b)$ a  conformal algebra. Then $B' = (B, [\cdot_\l\cdot], *_\l,\a,\b)$ is a regular BiHom-Poisson conformal algebra. Where $[x_\l y] = x *_\l y - \a^{-1}\b(y)*_{-\p-\l} \a\b^{-1}(x),$ for all $x, y\in B$ and $\l\in \mathbb{C}$.
	\end{prop}
%\begin{lem} \label{lem*} If $(B,\cdot_\l, [._\l.], \alpha, \beta)$ is a BiHom-Poisson conformal algebra, then we have the following equations: \[ [\beta(x)_\l (\beta(y)\cdot_\m\alpha(z))] = \beta^2(y)\cdot_\m[\beta(x)\cdot_\l \alpha(z)] + \beta^2(z)\cdot_\l[\beta(x)_\m \alpha(y)], \] \[ \beta^2(z)_{-\p-\l-\m}(\beta(x)_\l\alpha(y)) = \beta^2(y)_\m (\beta(z)_{-\p-\l}\alpha(x)) = \beta^2(x)_\l (\beta(y)_\m\alpha(z)), \] \[ \alpha\beta(x)\cdot_\l(y\cdot_\m z) = \beta(y)\cdot_\m(\alpha(x)\cdot_\l z) = (\beta(x)\cdot_\l y)\cdot_{\l+\m}\beta(z). \] \end{lem}
\begin{prop}
	Let $(B , *_{\l}, [\cdot_\l\cdot], \alpha, \beta)$ be a BiHom-Poisson conformal algebra.
	Then $(B , *_{\l}^n=\a^n\circ *_\l, [\cdot_\l\cdot]^n=\a^n[\cdot_\l\cdot], \alpha^{n+1}, \a^n\beta)$ is a BiHom-Poisson conformal algebra.
\end{prop}
\begin{proof}
	\begin{enumerate}
	\item Here we show that $(B , *_{\l}^n=\a^n\circ *_\l, \alpha^{n+1}, \alpha^n\beta)$ is a  conformal algebra, for this consider 
	\begin{itemize}
		\item Conformal sesqui-linearity:\begin{align*}
		\p(x)*_\l^n y=& \a^n(\p(x)*_\l y)= \a^n(-\l(x*_\l y))\\&= -\l\a^n(x*_\l y)=-\l(x*_\l^n y)
		\end{align*}
		\begin{align*}
		x*_\l^n \p(y)=& \a^n(x*_\l \p(y))= \a^n((\l+\p)(x*_\l y))\\&= (\l+\p)\a^n(x*_\l y)=(\l+\p)(x*_\l^n y)
		\end{align*} 
		\item BiHom-skew symmetry:
		\begin{align*}
		\a^n\b(x)*_\l^n \a^{n+1}(y)&=\a^n(\a^n\b(x)*_\l \a^{n+1}(y))=\a^{2n}(\b(x)*_\l \a(y))\\&=-\a^{2n}(\b(y)*_{-\p-\l}\a(x))=-\a^{n}(\a^n\b(y)*_{-\p-\l}\a^{n+1}(x))=-(\a^n\b(y)*^{n}_{-\p-\l}\a^{n+1}(x))
		\end{align*}
		\item Conformal  identity:
		\begin{align*}
		&\a^{n+1}(x)*_\l^n(y*_\m^n z)-(x*_\l^ny)*_{\l+\m}^n \a^n\b(z)\\&=  
		\a^n(\a^{n+1}(x)*_\l\a^n(y*_\m z))-\a^n(\a^n(x*_\l y)*_{\l+\m} \a^n\b(z))\\&=  
		\a^{2n+1}(x)*_\l\a^{2n}(y*_\m z)-\a^{2n}(x*_\l y)*_{\l+\m} \a^{2n}\b(z)\\&=  
		\a^{2n}(\a (x)*_\l(y*_\m z)-(x*_\l y)*_{\l+\m}\b(z))\\&  
		=0.
		\end{align*}
	\end{itemize}
	Thus, $(B , *_{\l}^n=\a^n\circ *_\l, \alpha^{n+1}, \alpha^n\beta)$ is a  conformal algebra.
	\item Similarly, we can show that $(B , [._\l. ]^n=\a^n[._\l. ], \alpha^{n+1}, \alpha^n\beta)$ is a BiHom-Lie conformal algebra.
	 \item  Now we only left to prove BiHom-Leibniz conformal identity as follows:
	 \begin{align*}&([\a^n\b(x)_\l y]^n)*_{\l+\m}^n \a^n\b(z) + \a^n\b(y)*_{\m}^n ([\a^{n+1}(x)_\l z]^n)\\&= \a^n((\a^n[\a^n\b(x)_\l y])*_{\l+\m} \a^n\b(z)) + \a^n(\a^n\b(y)*_{\m}  \a^n([\a^{n+1}(x)_\l z])), \\&= \a^{2n}([\a^n\b(x)_\l y]*_{\l+\m} \b(z)) + \a^{2n}(\b(y)*_{\m}  [\a^{n+1}(x)_\l z]),\\&= \a^{2n}([\a^n\b(x)_\l y]*_{\l+\m} \b(z) + \b(y)*_{\m}  [\a^{n+1}(x)_\l z]),
	 \\&= \a^{2n}([\a \b(\a^n(x))_\l (y *_{\m}  z)])\\&=\a^{2n}[\a^{n+1}\b(x)_\l ( y*_\m z) ] \\&=\a^n[\a^{2n+1}\b(x)_\l \a^n( y*_\m z) ] \\&=[\a^{n+1}\a^n\b(x)_\l ( y*_\m^n z) ]^n
	 \end{align*} 
 \end{enumerate}It completes the proof.
	\end{proof} In the following proposition, we show the relation between Rota-Baxter operator and BiHom-Poisson conformal algebra.
\begin{prop}
	Let $(B , *_{\l}, [._\l. ], \alpha, \beta)$ be a BiHom-Poisson conformal algebra and $R$ is a Rota-Baxter operator. Define two new $\l$-multiplication maps $*'_\l$ and $ [._\l.]' $ with
	$x *'_\l y = R(x)*_\l y + x*_\l R(y),
	[x_\l y]' = [R(x)_\l y] + [x_\l R(y)].$
	Then $(B , *'_{\l}, [\cdot_\l\cdot]', \alpha, \beta)$ is a BiHom-Poisson conformal algebra.
\end{prop}
\begin{proof} Here we show that $(B , [\cdot_\l\cdot]', \alpha, \beta)$ is a BiHom-Lie conformal algebra.  \begin{itemize}
		\item Conformal sesqui-linearity:\begin{align*}
	[\p(x)_\l y]'=& [R(\p (x))_\l y] + [\p (x)_\l R(y)]= [\p (R(x))_\l y] + [\p (x)_\l R(y)]\\&= -\l[ R(x)_\l y] -\l [x_\l R(y)]=-\l([ R(x)_\l y] + [x_\l R(y)])=-\l([x_\l y]')
		\end{align*}
	Similarly we can show $[x_\l \p(y)]'=(\l+\p)([x_\l y]')$.
	\item For BiHom-skew-symmetry:
	\begin{align*}
[\b(x)_\l \a(y)]'&= [ R(\b (x))_\l \a (y)] + [\b (x)_\l R(\a (y))]\\&=  [\b ( R(x))_\l \a (y)] + [\b (x)_\l \a(R( y))]\\&=  - [\b (y)_{-\p-\l} \a (R(x))]- [\b (R( y))_{-\p-\l }\a(x)]\\&=  - ([\b (y)_{-\p-\l} R (\a(x))] + [R (\b( y))_{-\p-\l }\a(x)])=  - ([\b (y)_{-\p-\l}   \a(x)]')
	\end{align*}
\item BiHom-Jacobi identity is given as follows:
\begin{align*}&[[\b(x)_\l y]'_{\l+\m}\b(z)]' + [\b(y)_\m[\a(x)_\l z]']'\\&= [([R(\b(x))_\l y]+ [\b(x)_\l R(y)])_{\l+\m}\b(z)]' + [\b(y)_\m([R(\a(x))_\l z]+ [\a(x)_\l R(z)])]'\\&= ([R([R(\b(x))_\l y]+ [\b(x)_\l R(y)])_{\l+\m}\b(z)]+ [([R(\b(x))_\l y]+ [\b(x)_\l R(y)])_{\l+\m}R\b(z)]) \\&+( [R\b(y)_\m([R(\a(x))_\l z]+ [\a(x)_\l R(z)])]+ [\b(y)_\m R([R(\a(x))_\l z]+ [\a(x)_\l R(z)])])\\&= [( [R(\b(x))_\l R(y)])_{\l+\m}\b(z)]+ [([R(\b(x))_\l y]+ [\b(x)_\l R(y)])_{\l+\m}R\b(z)] \\&+( [R\b(y)_\m([R(\a(x))_\l z]+ [\a(x)_\l R(z)])]+ [\b(y)_\m  [R(\a(x))_\l R(z)]])\\&= [( [\b(R(x))_\l R(y)])_{\l+\m}\b(z)]+ [[\b(R(x))_\l y]_{\l+\m}\b(R(z))]+ [[\b(x)_\l R(y)]_{\l+\m}\b(R(z))] \\&+ [\b(R(y))_\m[\a(R(x))_\l z]]+ [\b(R(y))_\m[\a(x)_\l R(z)]]+ [\b(y)_\m  [\a(R(x))_\l R(z)]]\\&= [\a\b(R(x))_\l[ R(y) _\m  z]]+ [\a\b R(x)_\l  [y_\m R(z)]]+ [\a\b(x)_\l[ R(y)_\m R(z)]]\\&=[\a\b(R(x))_\l[R(y)_\m z]]+[\a\b(R(x))_\l[y_\m R (z)]] +[\a\b(x)_\l R([R(y)_\m z]+[y_\m R (z)])]\\&=[\a\b(R(x))_\l[R(y)_\m z]]+[\a\b(x)_\l R([R(y)_\m z])] + [\a\b(R(x))_\l[y_\m R (z)]]+[\a\b(x)_\l R([y_\m R (z)])]\\&=[R(\a\b(x))_\l[R(y)_\m z]]+[\a\b(x)_\l R([R(y)_\m z])] + [R(\a\b(x))_\l[y_\m R (z)]]+[\a\b(x)_\l R([y_\m R (z)])]\\&=[\a\b(x)_\l[R(y)_\m z]]'+ [\a\b(x)_\l[y_\m R (z)]]'[\a\b(x)_\l[y_\m z]']'\\& =[\a\b(x)_\l([R(y)_\m z]+ [y_\m R (z)])]'=[\a\b(x)_\l[y_\m z]']'
\end{align*}	
	\end{itemize} 
Similarly we can show $(B , *'_{\l}, \alpha, \beta)$ is BiHom-associative conformal algebra. Thus, we only left to varify BiHom-Leibniz conformal identity. It can be seen as  $ $   
\begin{align*} [\a\b(x)_\l ( y*'_\m z) ]' & 
 =[R(\alpha\beta(x))_\l (R(y)*_\m z + y*_\m R(z))] + [\alpha\beta(x)_\l R(R(y)*_\m z + y*_\m R(z))] \\
& = [R(\alpha\beta(x))_\l ( R(y)*_\m z)] + [R(\alpha\beta(x))_\l (y*_\m R(z))] + [\alpha\beta(x)_\l (R(y)*_\m R(z))] \\
& = [R(\beta(x))_\l  R(y)]*_{\l+\m } \beta(z)  + R(\beta(y))*_\m[R(\alpha(x))_\l z] + [R\beta(x)_\l y]*_{\l+\m}R\beta(z) \\
& + \beta(y)*_\m [R\alpha(x)_\l R(z)] + [\beta(x)_\l R(y)]*_{\l+\m}R\beta(z) + R\beta(y)*_\m[\alpha(x)_\l R(z)]\\&= ([\b(Rx)_\l Ry])* _{\l+\m} \b(z)+ [\b(Rx)_\l y]* _{\l+\m} \b(Rz)+[\b(x)_\l Ry]* _{\l+\m} \b(Rz) \\&+ \b(Ry)*_\m [\a(Rx)_\l z]+ \b(Ry)*_\m[\a(x)_\l Rz] +\b(y)*_\m ([\a(Rx)_\l Rz]),\\&=
R([\b(Rx)_\l y]+[\b(x)_\l Ry])* _{\l+\m} \b(z)+ ([\b(Rx)_\l y]+[\b(x)_\l Ry])* _{\l+\m} \b(Rz) \\&+ \b(Ry)*_\m ([\a(Rx)_\l z]+ [\a(x)_\l Rz])+\b(y)*_\m R([\a(Rx)_\l z]+ [\a(x)_\l Rz]),\\&=([\b(x)_\l y]')*'_{\l+\m} \b(z) + \b(y)*'_\m ([\a(x)_\l z]')
\end{align*}
It completes the proof. \end{proof}
Now we define the tensor product of BiHom-Poisson conformal algebra as follows:
\begin{lem}   
		Let $(B_1, *_{1\l}, [\cdot_\l\cdot]_1, \alpha_1, \beta_1,\p_1)$ and $(B_2, *_{2\l}, [\cdot_\l\cdot]_2, \alpha_2, \beta_2,\p_2)$ are two $ (non-commutative) $ BiHom-Poisson conformal algebras. Define two linear maps $\alpha, \beta: B_1 \oplus B_2 \rightarrow B_1 \oplus B_2$ such that 
		$$\a=\a_1\otimes \a_2 ~~and~~ \b=\b_1\otimes \b_2$$ and two $\l$-multiplication maps $*_{\oplus\l}, [\cdot_{\oplus\l}\cdot]: (B_1 \oplus B_2) \otimes (B_1 \oplus B_2) \rightarrow (B_1 \oplus B_2)[\l]$ such that the following conditions hold for all $x,y \in B_1$ , $p,q \in B_2$ and $\l\in \mathbb{C}:$
		\[
		[(x + p)_{\oplus\l} (y + q)]= [x_\l y]_1  +  [p_\l q ]_2,
		\]
			\[ (x + p)_{\oplus\l}(y + q)=  ( p *_{2\l} q) + ( x*_{1\l} y ) ,
		\]
		
		Then $(B_1 \oplus B_2, *_{\oplus\l}, [\cdot_{\oplus\l}\cdot], \alpha, \beta)$ is a BiHom-Poisson conformal  algebra.
\end{lem}
\begin{proof}In order to show  $(B_1 \otimes B_2, \cdot_{\oplus\l}, [\cdot_{\oplus\l}\cdot], \alpha, \beta, \p_\oplus)$ is a BiHom-Poisson conformal algebra. 
	\begin{enumerate}
		\item We first show that  $(B_1 \otimes B_2, [\cdot_{\oplus\l}\cdot], \alpha, \beta,\p_\oplus)$ is a BiHom-Poisson conformal algebra. For this we need to satisfy the following identities:\begin{itemize}
			\item Conformal sesqui-linearity:	\begin{align*}
			[\p_\oplus(x+ p)_{\oplus\l} (y + q)] &= [\p_1(x) + \p_2(p)_\l (y + q)]\\&= [\p_1(x) _\l (y)]+ [\p_2(p) _\l (q)]\\&=-\l[x _\l (y)]_1-\l[p _\l (q)]_2\\&=-\l([x _\l (y)]_1+[p _\l (q)]_2)\\&=-\l([(x + p)_{\oplus\l} (y + q)])
			\end{align*} Similarly, we can show that \begin{align*}
			[(x + p)_{\oplus\l} \p(y + q)]&=(\l+\p)([(x + p)_{\oplus\l} (y + q)])
			\end{align*}
			\item  For BiHom-skew symmetry:
			\begin{align*}
			[\beta(x + p)_{\l\oplus} \alpha(y + q)]
			&= [\beta_1(x) + \beta_2(p)_{\l\oplus} \alpha_1(y) + \alpha_2(q)] \\
			&= [\beta_1(x)_\l \alpha_1(y)]_1  +  [\beta_2(p)_\l \alpha_2(q)]_2 \\
			&= -[\beta_1(y)_{-\p-\l} \alpha_1(x)]_1 -  [\beta_2(q)_{-\p-\l} \alpha_2(p)]_2\\
			&= -[\beta_1(y) + \beta_2(q)_{\oplus(-\p-\l)} \alpha_1(x) + \alpha_2(p)] \\
			&= -[\beta(y + q)_{\oplus(-\p-\l)} \alpha(x + p)].
			\end{align*} Here we have used the fact that $B_1$ and $B_2$ are Lie conformal algebras and satisfy skew symmetric identity.
			\item For BiHom-Jacobi Identity, We compute that for $x,y,z \in B_1$ and $p,q,r \in B_2$:
			\begin{align*}
			\circlearrowright_{x,y,z}^{p,q,r}
			[\beta^2(x + p)_{\oplus\l} [\beta(y + q)_{\oplus\m} \alpha(z + r)]] 
			&= \circlearrowright_{x,y,z}^{p,q,r}
			[\beta^2_1(x)+ \beta^2_2(p)_{\oplus\l} [\beta_1(y)+ \beta_2(q)_{\oplus\m} \alpha_1(z) + \alpha_2(r)]] 
			\\&= \circlearrowright_{x,y,z}^{p,q,r} [\beta_1^2(x)_\l [\beta_1(y)_\m \alpha_1(z)]_1]_1 
			+[\beta_2^2(p)_\l [\beta_2(q)_\m\alpha_2(r)]_2]_2\\&=0.
			\end{align*}
		\end{itemize}
	\item It is easy to show that $(B_1 \oplus B_2,*_{\oplus\l}, \alpha, \beta)$ is a  conformal algebra. 
	\item Now we show that BiHom-Leibniz conformal identity holds, for this consider that:
	\begin{align*}
	& ([\b(x+ p)_{\oplus\l} (y+ q) ])*_{\oplus(\l+\m)} \b(z+ r) + \b(y+ q)*_{\oplus\m} ([\a(x+ p)_{\oplus\l} (z+ r)])\\& = ([(\b_1(x)+ \b_2(p))_{\oplus\l} (y+ q) ])*_{\oplus(\l+\m)} (\b_1(z)+\b_2(r)) + (\b_1(y)+ \b_2(q))*_{\oplus\m} ([(\a_1(x)+ \a_2(p))_{\oplus\l}z+ r])
	\\&  = ([\b_1(x)_\l y ]_1+[\b_2(p)_\l q ]_2)*_{\oplus(\l+\m)} (\b_1(z)+ \b_2(r)) + (\b_1(y)+\b_2(q))*_{\oplus\m} ([\a_1(x) _\l z ]_1+[\a_1(p) _\l r ]_2)\\&  = ([\b_1(x)_\l y ]_1)*_{1(\l+\m)} \b_1(z)+([\b_2(p)_\l q ]_2)*_{2(\l+\m)} \b_2(r) + \b_1(y)*_{1\m} ([\a_1(x) _\l z ]_1)+ \b_2(q)*_{2\m}([\a_2(p) _\l r ]_2)\\&  = ([\b_1(x)_\l y ]_1)*_{1(\l+\m)} \b_1(z)+ \b_1(y)*_{1\m} ([\a_1(x) _\l z ]_1) +([\b_2(p)_\l q ]_2)*_{2(\l+\m)} \b_2(r)+  \b_2(q)*_{2\m}([\a_2(p) _\l r ]_2)\\&= [\a_1\b_1(x)_\l(y*_{1\m} z)]_1+[\a_2\b_2(p)_\l(q*_{2\m }r)]_2\\&= [\a\b(x+p)_{\oplus\l}((y+q)*_{\oplus\m} (z+r))].
	\end{align*} 
	\end{enumerate}It completes the proof.
	\end{proof}
\begin{thm}\label{thmmor}
	Consider that $(B_1, *^1_{\l}, [\cdot_\l\cdot]_1, \alpha_1, \beta_1,\p_1)$ is a BiHom-Poisson conformal algebra, where $\alpha_1, \beta_1$ are structural maps. Assume that there exist $\phi_1,\psi_1\in gc(B_1)$ such that $\a_1,\b_1 , \phi_1,\psi_1$ commute. In this case $B'_1=(B_1, *'^{1}_{\l}=(\phi\otimes\psi)\circ *^1_{\l},[\cdot_\l\cdot]'_1= (\phi\otimes\psi)\circ [\cdot_\l\cdot]'_1, \alpha_1\phi, \beta_1\psi,\p_1)$ is a BiHom-Poisson conformal algebra.\\
	Now consider that,  $(B_2, *^2_{\l}, [\cdot_\l\cdot]_2, \alpha_2, \beta_2,\p_2)$ is another BiHom-Poisson conformal algebra  and there exists $\phi_2,\psi_2$ such that $\alpha_2, \beta_2,\phi_2,\psi_2$ commute.\\ 
	 Let the $f$ is the morphism of these BiHom-Poisson conformal algebras, given by $$ (B_1, *^1_{\l}, [\cdot_\l\cdot]_1, \alpha_1, \beta_1,\p_1)\to  (B_2, *^2_{\l}, [\cdot_\l\cdot]_2, \alpha_2, \beta_2,\p_2) $$ such that $f\phi=\phi' f$ and  $f\psi=\psi' f$. In this scenario $f : B_1'\to B_2'$ is also a morphism.
\end{thm}
\begin{proof}
	To show $B'_1=(B_1, *'^1_{\l}=(\phi\otimes\psi)\circ *^1_{\l}, [\cdot_\l\cdot]'_1= (\phi\otimes\psi)\circ [._\l. ]'_1, \alpha_1\phi, \beta_1\psi,\p_1)$ is a BiHom-Poisson conformal algebra, we only show the associative  conformal algebra case, other cases can be proved similarly.
	Note that \begin{itemize}
		\item Conformal sesqui-linearity:
		\begin{align*}
		\p(x)*'^1_{\l} y&= (\phi\otimes \psi) (\p(x)*^1_{\l} y)= (\p(\phi (x))*^1_{\l} \psi(y))\\&= -\l( \phi (x) *^1_{\l} \psi(y))=-\l( \phi \otimes \psi)(x  *^1_{\l} y)= -\l (x*'^1_{\l} y)
		\end{align*}
		\begin{align*}
		x*'^1_{\l} \p(y)=& (\phi\otimes \psi)(x*^1_\l \p(y))= \phi(x)*^1_\l \p(\psi( y))\\&= (\l+\p)\phi(x)*^1_{\l} \psi(y)=(\l+\p)(x*'^1_{\l} y).
		\end{align*} 
		\item BiHom-skew symmetry:
		\begin{align*}
		\b_1(\psi(x))*'^1_{\l} \a_1(\phi(y))&=\b_1(\phi\psi(x))*^1_{\l} \a_1(\phi\psi(y))\\&=\b_1(\phi\psi(y))*^1_{-\p-\l} \a_1(\phi\psi(x))
		=\b_1\psi(y)*'^1_{ -\p-\l } \a_1\phi(x).
		\end{align*}
		\item  Associative conformal identity:
		\begin{align*}
		\a\phi(x)*'^1_{\l} (y*'^1_{\m} z) &=\a\phi\phi(x)*^1_{\l} \psi(\phi (y)*^1_{\m} \psi (z))
		=\a\phi^2(x)*^1_{\l} (\psi\phi (y)*^1_{\m} \psi^2 (z))\\&=(\phi^2 (x)*^1_{\l}\phi\psi (y))*^1 _{\l+\m} \b\psi^2(z)=\phi(\phi (x)*^1_{\l}\psi (y))* ^1_{\l+\m} \psi\b\psi(z)\\&=(x*'^1_{\l}y)*'^1_{ \l+\m } \b\psi(z).
		\end{align*}
	\end{itemize}Thus, $B'_1=(B_1, *'^1_{\l}=(\phi\otimes\psi)\circ *^1_{\l} , \alpha_1\phi, \beta_1\psi,\p_1)$ is an associative  conformal algebra. Now, we show that there is an algebra morphism from $f : B_1'\to B_2'$. Assume that $x,y\in B'_1$ and $\l,\m\in \mathbb{C}$, we have 
\begin{align*}
f[x_\l y]'_1 &= f[\phi (x)_\l \psi (y)]_1
= [f(\phi (x))_\l f(\psi (y))]_2\\&= [\phi'f(x)_\l \psi'f( y)]_2=(\phi'\otimes \psi')\circ[f( x)_\l f( y)]_2\\&= [f( x)_\l f( y)]'_2.
\end{align*} Similarly, we get
\begin{align*}
f( x*'^1_{\l} y) &=  (f( x)*'^2_{\l} f( y)).
\end{align*}This completes the proof.
\end{proof}
From the previous theorem, we can get the following two results:
\begin{cor}
Let $(B , *_{\l}, [\cdot_\l\cdot], \alpha, \beta)$ be a BiHom-Poisson conformal algebra.  then  $B^k =(B , *^k_{ \l}=(\a^k\otimes\b^k)\circ *_{ \l} , [\cdot_\l\cdot]^k= (\a^k\otimes\b^k)\circ[\cdot_\l\cdot], \alpha^{k+1} , \beta^{k+1})$  is also a  BiHom-Poisson conformal algebra.
\end{cor} \begin{proof} Here the proof can be completed by using theorem \ref{thmmor} and replacing  $ \phi,\psi $ by $ \a^k ,\b^k $ respectively.
\end{proof}
\begin{cor}
	Assume that $(P , *_{\l}, [\cdot_\l\cdot])$ be a Poisson conformal algebra. Let $\a,\b$ are linear endomorphisms of $P$ then $B' =(P , *'_{ \l}=(\a\otimes\b )\circ *_{ \l} , [\cdot_\l\cdot]' = (\a\otimes\b )\circ[\cdot_\l\cdot ], \alpha , \beta)$ is another associated  BiHom-Poisson conformal algebra.
\end{cor}\begin{proof}The proof is followed by considering $\a=\b=id$ in Theorem \ref{thmmor}.\end{proof} Now we introduce the conformal representations of BiHom-Poisson conformal algebra as follows
	\begin{defn}A conformal representation of a BiHom-Poisson conformal algebra $ (B, *_\l, [._\l.], \a,\b) $ on a $ \mathbb{C}[\p] $-module $M$ endowed with a linear maps $\phi,\psi : M\to M$ such that $ \phi\p = \p\phi $ and $\p\psi=\psi\p$ is a tuple $(M,l,r, \rho,\phi,\psi) $, where $(M,l,r,\phi,\psi)$ is a conformal representation of the  conformal algebra $ (B,*_\l, \a,\b) $ and $(M,\rho,\phi,\psi)$ is a conformal representation of the BiHom-Lie conformal algebra $ (B,[._\l.], \a,\b) $, such that for all $x,y\in B$, $m\in M$ and $\l,\m\in \mathbb{C}$ we have the following equations
		\begin{eqnarray}
		\label{eq17}&l([\b(x)_\l y])_{\l+\m}\psi(m) = \rho(\a\b(x))_\l (l(y)_\m m) - l(\b(y))_\m{\rho(\a(x))}_\l m\\&r([\a(x)_\l y])_{\l+\m}\psi(m) = \rho(\a\b(x))_\l (r(y)_\m m) - r(\b(y))_\m{\rho(\b(x))}_\l m\\& \rho(x _\l y)_{\l+\m}\phi\psi(m) = l(\a(x))_\l (\rho(y)_\m\phi(m)) - l(\a(y))_\m (\rho(x)_\l \psi(m)). 
		\end{eqnarray}
	\end{defn}
	%NEED to confirm if the last equation valid/correct.
	\begin{prop}Let $(l,r, \rho , \phi, \psi, M)$ be a conformal representation of a BiHom-Poisson conformal algebra $(B, *_\l,[._\l.], \a, \b)$, where $M$ is a $ \mathbb{C}[\p] $-module, $\phi,\psi$ are  $ \mathbb{C} $-linear maps. Moreover, assume that $\a,\psi $ are bijective maps. Then, the semi direct product $B\oplus M$ of vector spaces is turned into a BiHom-Poisson conformal algebra $(B\oplus M, *'_\l,[._\l.]', \a, \b)$,by the $\l$-multiplication in $B \oplus M$ defined as follows
		\begin{equation*}
		\begin{aligned}(x_1 +m_1)*'_\l(x_2 + m_2) &:= x_1 *_\l x_2 + (l(x_1)_\l m_2 + r(x_2)_{\l} m_1),\\ [(x_1 +m_1)_\l(x_2 + m_2)]' &:= [{x_1} _\l x_2] + ({\rho(x_1)}_\l m_2 - {\rho(\a^{-1}\b (x_2))}_{-\p-\l}\phi\psi^{-1}( m_1)),\\
		(\a \oplus \phi)(x_1 + m_1) &:= \a(x_1) + \phi(m_1),\\
		(\b \oplus \psi)(x_1 + m_1) &:= \b(x_1) + \psi(m_1),
		\end{aligned}
		\end{equation*}for all $x_1, x_2 \in B, m_1, m_2 \in M$ and $\l,\m\in \mathbb{C}$. 
	\end{prop}
	\begin{proof}To show $B\oplus M$ is a space of BiHom-Poisson conformal algebra, we need to satisfy the axioms of Definition \ref{defBPA}. First axiom is straight forward to show and second axiom can be seen in Proposition 3.1 of \cite{C}. Now we  only  left to prove the BiHom-Leibniz conformal identity. For all $x_1, x_2, x_3\in B$,  $m_1, m_2, m_3 \in M$ and $\l, \m\in \mathbb{C}$ we have  
		\begin{align*}&[{(\a\b + \phi\psi)(x_1 + m_1)}_{\l}((x_2 + m_2) *'_\m (x_3 + m_3))]'\\&
		- [{(\b + \psi)(x_1 + m_1)}_\l (x_2 + m_2)]' {*'}_{\l+\m} (\b + \psi)(x_3 + m_3)\\&
		- {(\b+\psi)(x_2 + m_2)}{*'}_\m [{(\a + \phi)(x_1 + m_1)}_\l (x_3 + m_3)]'\\
		=& [{(\a\b(x_1) +\phi\psi(m_1))}_\l (x_{2} *_\m x_{3} + {l(x_{2})}_\m m_{3} +{ r(x_{3})}_\m m_{2})]'
	\\&- [{(\b(x_1) + \psi(m_1))}_\l (x_2 + m_2)]' {*'}_{\l+\m} (\b(x_3) + \psi (m_3))\\&
	- (\b(x_2)+\psi(m_2)){*'}_\m[{(\a(x_1) + \phi( m_1))}_\l (x_3 + m_3)]'\\
	=& [{(\a\b(x_1) +\phi\psi(m_1))}_\l (x_{2} *_\m x_{3} + {l(x_{2})}_\m m_{3} + {r(x_{3})}_{\m} m_{2})]'\\&- ([{\b(x_1)}_\l x_2]+{\rho(\b(x_1))}_\l m_2-{\rho(\a^{-1}\b(x_2))}_{-\p-\l}\phi(m_1)) {*'}_{\l+\m} (\b(x_3) + \psi (m_3))\\&-(\b(x_2)+\psi(m_2)){*'}_\m {[{\a(x_1)}_\l (x_3)]+{\rho(\a(x_1))}_\l m_3-{\rho(\a^{-1}\b(x_3))}_{-\p-\l}\phi^2\psi^{-1}(m_1)}
	\\=&[{\a\b(x_1)}_\l({x_2} *_\m x_3)]+{\rho(\a\b(x_1))}_{\l}{l(x_2)_\m m_3}+{\rho(\a\b(x_1))}_\l({r(x_3)}_{\m} m_2)- \rho(\a^{-1}\b({x_2}_\m x_3))_{-\p-\l}\phi^2(m_1)
	\\&-[{\b(x_1)}_\l x_2]*_{\l+\m}\b(x_3)-l{([{\b(x_1)}_\l x_2])}_{\l+\m}\psi(m_3) \\&-{r(\b(x_3))}_{-\p-\m}\rho(\b(x_1))_\l m_2+ r(\b(x_3))_{-\p-\m}\rho(\a^{-1}\b(x_2))_{-\p-\l} \phi (m_1)
	\\&-\b(x_2)*_\m [{\a(x_1)}_\l x_3]-l(\b(x_2))_\m(\rho(\a(x_1))_\l m_3)\\&+ l(\b(x_2))_\m(\rho(\a^{-1}\b(x_3))_{-\p-\l}\phi^2\psi^{-1}(m_1))+ r([{\a(x_1)}_\l x_3])_{-\p-\m}\psi(m_2)
	\\=&\left([{\a\b(x_1)}_\l ({x_2}*_\m {x_3})] - [{\b(x_1)}_\l {x_2}]*_{\l+\m} \b(x_3) - \b(x_2)*_\m[{\a(x_1)}_\l {x_3}]\right)\\&+\left({\rho(\a\b(x_1))}_\l {l(x_2)}_\m{(m_3)} - l([{\b(x_1)}_\l {x_2}])\psi(m_3)- l(\b(x_2))_\m{{\rho(\a(x_1))}_\l(m_3)}\right)\\&+\left({\rho(\a\b(x_1))}_\l {{r(x_3)}_{\m}(m_2)} - {r(\b(x_3))}_{-\p-\m}{\rho(\b(x_1))}_\l {(m_2)} -{r([{\a(x_1)}_\l {x_3}])}_{-\p-\m}{\psi(m_2)}\right)\\&- \left({\rho(\a^{-1}\b({x_2}*_\m  {x_3}))}_{-\p-\l}\phi^{2}(m_1)+ {r(\b(x_3))}_{-\p-\m}{\rho(\a^{-1}\b(x_2))}_{-\p-\l}{\phi(m_1)}+ {l(\b(x_2))}_\m {\rho(\a^{-1}\b(x_3))_{-\p-\l}\phi^2\psi^{-1}(m_1)}\right)\\=&0. \end{align*}This finishes the proof.\end{proof}
\begin{ex}
Let $ (B, *_{\l}, [._\l.],\a,\b) $ be a  BiHom-Poisson conformal algebra. Then
$(l, r, ad, \a, \b, B) $ is a regular representation of $ B $, where $l(x)_\l y = x*_\l y, r(x)_\l y =y*_{-\p-\l} x$ and $ ad(x)_\l y = [x_\l y], $ for all $x, y \in B$ and $\l\in \mathbb{C}$.\end{ex}
	\section{BiHom-pre-Poisson conformal algebra and its conformal  Bimodule}
In this section we first introduce the notion and conformal representation of BiHom-preLie conformal algebra, that lead us to describe (noncommutative) BiHom-pre-Poisson conformal algebra and discussed conformal bimodule structure on it.
\begin{defn} A BiHom-preLie conformal algebra $ (A, *_\l, \a, \b)$ is tuple consisting of a $ \mathbb{C}[\p] $- module $B$, a $\mathbb{C}$-bilinear map $ *_\l: B\times B \to B[\l]$ , two commutative multiplicative linear maps $\a,\b:B\to B$ such that $\p\a=\a\p$,$\b\p=\p\b$, satisfying the following equation for all $x,y,z \in B$ and $\l,\m\in \mathbb{C}$
\begin{equation*}(\b(x) *_\l\a(y)) *_{\l+\m}\b(z)-\a\b(x) *_\l (\a(y) *_\m z) = (\b(y) *_\m\a(x)) *_{\l+\m}\b(z)-\a\b(y) *_\m (\a(x) *_{\l+\m} z).\end{equation*} 
\end{defn}If $B$ is finitely generated, then  BiHom-preLie  conformal algebra is called finite. \begin{prop}\label{prop3.2}
Let $ (B, *_\l, \a, \b)$ be a regular BiHom-preLie conformal algebra with bijective structure map $\a$ and $\b$. Then $ (B,[ ._\l.], \a, \b)$ is called BiHom-Lie conformal algebra with the $\l$- bracket given by 
\begin{equation*}
[x_\l y] = x *_\l y -\a^{-1}\b(y) *_{-\p-\l} \a\b^{-1}(x),
\end{equation*}
for  all $x, y \in B$. We call this algebra  $B^c= (B,[ ._\l.], \a, \b)$ a sub-adjacent BiHom-Lie conformal algebra of $(B, *_\l,\a,\b)$.
\end{prop}Now, we introduce the notion of conformal representation of a BiHom-preLie conformal algebra in the following definition.
\begin{defn} Let $ (B, *_\l, \a,\b) $ be a BiHom-preLie conformal algebra, and let $(M, \phi, \psi)$ be a BiHom-conformal module. Let $l_*, r_* : B \to gc(M)$ are two $\mathbb{C}$-linear maps. The tuple $(l_*, r_*,\phi,\psi, M) $ is called a conformal representation of BiHom-preLie conformal algebra $B$, if for all $ x, y \in B, m \in M $ and $\l,\m\in \mathbb{C}$, we have \begin{align}&l_*(\p x)_\l m = -\l(l_*(x)_\l  m),\\&
	l_*(x)_\l(\p m) = (\p+\l)(l_*(x)_\l m),\\&
	r_*(\p x)_\l m = -\l(r_*(x)_\l  m),\\&
	r_*(x)_\l(\p m)= (\p+\l)(r_*(x)_\l m),\\&
\phi(l_*(x)_\l m) = l_*(\a(x))_{\l}\phi(m),\\&
	\phi(r_*(x)_\l m) = r_*(\a(x))_{\l}\phi(m), \\&
	\psi(l_*(x)_\l m) = l_*(\b(x))_\l\psi(m), \\&
	\psi(r_*(x)_\l m) = r_*(\b(x))_\l\psi(m)\\&l_*([\b(x)_\l \a(y)])_{\l+\m}\psi(m) = {l_*(\a\b(x))}_\l l_*(\a(y))_\m m - l_*(\a\b(y))_\m l_*(\a(x))_\l m \\&
	r_*(\b(y))_\m \rho(\b(x))_\l \phi(m) = l_*(\a\b(x))_\l r_*(y)_\m \phi(m)
	- r_*(\a(x) *_\l y)_{\l+\m}\phi\psi(m)
	\end{align} where $[\b(x)_\l \a(y)] = \b(x)*_\l \a(y)-\b(y)*_{-\l-\p}\a(x)$ and $(\rho\circ \b)\phi = (l_*\circ \b)\phi-(r_* \circ \a)\psi.$
\end{defn}
	\begin{prop}
	Let $(B, *_\l, \a, \b)$ be a BiHom-preLie conformal algebra and $(l_*,r_* , \phi, \psi, M)$ be its conformal representation, where $M$ is a $ \mathbb{C}[\p]$-module, $\phi,\psi$ are $\mathbb{C} $-linear maps, satisfying $\p\phi=\phi\p$ and $\p\psi=\psi\p$. Then, the direct sum $B\oplus M$ of vector spaces is turned into a BiHom-preLie conformal algebra by defining $\l$-multiplication $*'_\l$ in $B \oplus M$ as follows
	\begin{equation}
	\begin{aligned}(x_1 +m_1)*'_\l(x_2 + m_2) &:= x_1 *_\l x_2 + (l_*(x_1)_\l{ m_2} + r_*(x_2)_{-\p-\l} m_1),\\
	(\a \oplus \phi)(x  + m) &:= \a(x ) + \phi(m),\\
	(\b \oplus \psi)(x  + m) &:= \b(x) + \psi(m ),
	\end{aligned}
	\end{equation}for all $x,x_1, x_2 \in B,m, m_1, m_2 \in M$ and $\l\in \mathbb{C}$.\end{prop}We denote this BiHom-preLie conformal algebra by $(B\oplus M, *'_\l, \a+\b, \phi+\psi)$, or simply $(B\ltimes_{l_*,r_*,\a,\b,\phi,\psi}M)$.
	\begin{prop} Consider a regular BiHom-preLie conformal algebra $ (B, *_\l,\a,\b) $ and let $ (l_*, r_*, \phi,\psi, M ) $ be a conformal representation of it in such a way that $ \phi $ is bijective. Let $ (B, [._\l.], \a,\b) $ be the sub-adjacent BiHom-Lie conformal algebra of $(B, *_\l, \a,\b).$ Then $(l_* - (r_* \circ \a\b^{-1})\phi^{-1}\psi, \phi, \psi, M )$ is a conformal representation of BiHom-Lie conformal algebra $ (A, [._\l.], \a,\b) $.
	\end{prop}

\begin{proof}To show that $(l_* - (r_* \circ \a\b^{-1})\phi^{-1}\psi, \phi, \psi, M)$ is a conformal representation of a BiHom-Lie conformal algebra $(B, [._\l.], \a,\b) $, we need to satisfy the axioms of Definition \ref{defrepBLA}. Let's check them one by one
	\begin{enumerate}\item First we show that $$ \rho(\p x)_\l m = -\l(\rho(x)_\l  m),$$ so
		\begin{align*}
	 (l_* - (r_* \circ \a\b^{-1})\phi^{-1}\psi)(\p(x))_\l m &=
	l_* (\p(x))_\l m- (r_* \circ \a\b^{-1}(\p(x)))_\l\phi^{-1}\psi (m)
	\\&=-\l(l_* (x)_\l m)+\l (r_* \circ \a\b^{-1}(x))_\l\phi^{-1}\psi (m)\\&=
	-\l((l_* - (r_* \circ \a\b^{-1} )_\l\phi^{-1}\psi)(x))_\l m.
		\end{align*}Similarly, we can show that $$ (l_* - (r_* \circ \a\b^{-1})\phi^{-1}\psi)(x)_\l \p(m)=(\p+\l)((l_* - (r_* \circ \a\b^{-1} )\phi^{-1}\psi)(x))_\l m.$$
		\item Next we show that $$ \phi(\rho(x)_\l m)=\rho(\a(x))_\l \phi(m)$$\begin{align*}
	 \phi (l_{*}(x)- r_{*}(\a\b^{-1}
		(x))\phi^{-1}\psi)_\l m &=\phi (l_{*}(x))_\l m- \phi (r_{*}(\a\b^{-1}
		(x))_\l\phi^{-1}\psi(m))\\&
		= (l_{*}(\a (x)))_\l \phi(m)- (r_{*}( \a\b^{-1}
		\a(x))_\l\phi\phi^{-1}\psi(m)) \\&= (l_{*}(\a (x)))_\l \phi(m)- (r_{*}( \a\b^{-1}
		\a(x))_\l \psi(m)) 
		\\&= (l_{*}- (r_{*}\circ \a\b^{-1} )\phi\psi^{-1})(\a (x))_\l \phi(m).\end{align*} 
		Similarly, we can show that $$ \psi (l_{*}(x)- r_{*}(\a\b^{-1}
		(x))\phi^{-1}\psi)_\l m= (l_{*}- (r_{*}\circ \a\b^{-1} )\phi\psi^{-1})(\b (x))_\l \psi(m).$$
		\item Finally, we show $\rho([\b(x)_\l y])_{\l+\m}\psi(m) = \rho(\a\b(x))_\l \rho(y)_{\m}m -\rho(\b(y))_{\m} (\rho(\a(x))_{\l}m)$, so
			\begin{align*}  
	&(l^*(\a\b(x)) - (r^* \a\b^{-1}(\a\b(x)) \phi^{-1}\psi))_\l (l^*(y) - (r^* \circ \a\b^{-1}(y))\phi^{-1}\psi) _\m (m) \\
	&\quad - (l^*(\b(y)) - (r^*  \a\b^{-1}(\b(y))\circ\phi^{-1}\psi))_\m (l^*(\a(x)) - (r^* \circ \a\b^{-1}(\a(x))\phi^{-1}\psi)_\l (m))\\
	&=l^*(\a\b(x)) _\l (l^*(y)_\m m) - (r^*  \a^2(x))_\l\phi^{-1}\psi (l^*(y)_\m m) \\
	&\quad - l^*(\a\b(x))_\l  (r^*  \a\b^{-1}(y))_\m\phi^{-1}\psi(m) 
	+ (r^*  \a^2(x))_\l \phi^{-1}\psi (r^* \a\b^{-1}(y))_\m \phi^{-1}\psi(m) \\ 
	&\quad - l^*(\b(y))_\m  l^*(\a(x))_\l m + l^*(\b(y))_\m (r^* \circ \a^2\b^{-1}(x))_\l \phi^{-1}\psi(m) \\
	&\quad + (r^* \circ \a(y))_\m\phi^{-1}\psi \circ l^*(\a(x))_\l m - (r^* \circ \a(y))_\m \phi^{-1}\psi \circ (r^* \circ  \a^2\b^{-1}(x))_\l\phi^{-1}\psi(m)\\
	&=l^*(\a\b(x))_\l (  l^*(y)_\m m )- (r^*\circ \a^2(x))_\l ( l^*(\a^{-1}\b(y))_\m\phi^{-1}\psi(m)) \\
	&\quad - l^*(\a\b(x))_\l  ( (r^*\circ \a\b^{-1}(y))_\m\phi^{-1}\psi(m)) + (r^*\circ \a^2(x))_\l  ((r^*(y))_\m \phi^{-2}\psi^2(m)) \\
	&\quad - l^*(\b(y))_\m (l^*(\a(x))_\l m) + l^*(\b(y))_\m((r^*\circ \a^{2}(x))_\l\phi^{-1}\psi(m)) \\
	&\quad + (r^*\circ \a(y))_\m   (l^*(\b(x))_\l\phi^{-1}\psi(m)) - (r^*\circ \a(y))_\m ((r^*\circ \a(x))_\l \phi^{-2}\psi^2(m)) \\ &= \{l^*(\a\b(x))_\l  (l^*(y)_\m m) - l^*(\b(y))_\m   (l^*(\a(x))_\l m)\} \\
	&\quad + \{-l^*(\a\b(x))_\l ((r^*\circ \a\b^{-1}(y))_\m \phi^{-1}\psi(m)) \quad + (r^*\circ \a(y))_\m  ( l^*(\b(x))_\l \phi^{-1}\psi(m)) \\
	&\quad - (r^*\circ \a(y))_\m (  (r^*\circ \a(x))_\l \phi^{-2}\psi^2(m))\} \\
	&\quad - \{(r^*\circ \a^2(x))_\l  (l^*(\a^{-1}\b(y))_\m \phi^{-1}\psi(m))\quad - (r^*\circ  \a^2(x))_\l ( (r^*(y))_\m\phi^{-2}\psi^2(m)) \\
	&\quad - l^*(\b(y))_\m ((r^*\circ  \a^2\b^{-1}(x))_\l \phi^{-1}\psi(m))\}
	\\
	&=l_*([\b(x)_\l y])_{\l+\m}\psi(m)- r_*(\a(x) *_\l \a\b^{-1} (y))_{\l+\m}\phi^{-1}\psi^2 (m) + r_*(y*_{-\p- \l}\a^2\b^{-1}(x))_{\l+ \m}\phi^{-1}\psi^2(m)\\&=l_*([\b(x)_\l y])_{\l+\m}\psi(m)- r_*(\a(x) *_\l \a\b^{-1}(y) - y *_{-\p-\l}\a^2\b^{-1} (x))_{\l+\m}\phi^{-1}\psi^2 (m)\\&= l_*([\b(x)_\l y])_{\l+\m}\psi(m) - r_*([\a(x)_\l \a\b^{-1} (y)])_{\l+\m}\phi^{-1}\psi^2 (m)\\&=\rho([\b(x)_\l y])_{\l+\m}\psi(m).\end{align*}
	\end{enumerate}    
	\end{proof} 
\begin{defn}\cite{AW2}A $5$-tuple $(B, \<_\l, \>_\l, \a, \b)$ equipping a $\mathbb{C}[\p]$-module $B$, bilinear multiplication maps $\<_\l, \>_\l: B \otimes B \to B[\l]$ and commuting linear maps $\a, \b : B \to B$ is said to be a BiHom-dendriform conformal algebra, if the following conditions hold:
	\begin{eqnarray}(\p x)\>_\l y= -\l (x \>_\l y), &x \>_\l (\p y)= (\l+ \p) (x \>_\l y),\\(\p x)\<_\l y= -\l (x\<_\l y), &x\<_\l (\p y)= (\l+ \p) (x \<_\l y),\\
	\a(x \<_\l y) = \a(x) \<_\l \a(y),& \a(x \>_\l y) = \a(x) \>_\l \a(y),\label{D1} \\
	\b(x \<_\l y)= \b(x) \<_\l \b(y), & \b(x \>_\l y) = \b(x) \>_\l \b(y),\label{D2}  \\
	(x \<_\l y) \<_{\l+\m} \b(z) =& \a(x) \<_\l (y \<_\m z + y \>_\m z),\label{D3} \\
	(x \>_\l y) \<_{\l+\m} \b(z) =& \a(x) \>_{\l} (y \<_\m r), \label{D4} \\
	\a(x) \>_\l (y \>_\m z) =& (x \<_\l y + x \>_\l y) \>_{\l+\m} \b(z),\label{D5} \end{eqnarray}
	for all $x,y,z \in B$ and $\l, \m\in \mathbb{C}$.\end{defn}
where \begin{equation*}
  x\cdot_\l y = x \<_\l y + x\>_\l y. 
\end{equation*} 
\begin{lem}\label{lem3.7} The tuple $(B,\cdot_\l= \<_\l+ \>_\l, \a, \b)$ is an BiHom-associative conformal algebra provided that $(B, \<_\l, \>_\l, \a, \b)$ is a BiHom-dendriform conformal algebra.\end{lem} Now we introduce the conformal representation of BiHom-dendriform conformal algebra.
\begin{defn}
	Let $(B,\<_\l,\>_\l, \a,\b)$ be a BiHom-dendriform conformal algebra, and $ M $ be a vector space. Let $l_\<, r_\<, l_\>, r_\>: B \to gc(M)$ are four $ \mathbb{C} $-linear maps  and $\phi,\psi : M\to M$ be six linear maps. Then the tuple $(l_\<, r_\<, l_\>, r_\>, \phi,\psi, M)$ is called a conformal representation of $B$ if the following equations hold for any $x, y\in B$, $m \in M$ and $\l,\m\in \mathbb{C}$:
	\begin{align}
	&l_{\<}(x \<_\l y)_{\l+\m}\phi(m) = l_\<(\a(x))_\l (l_\cdot(y)_\m m), \\&
	r_\<(\b(x))_\l (l_\<(y)_\m m) = l_\<(\a(y))_\m (r_\cdot(x)_\l m), \\&
	r_\<(\b(x))_\l (r_\<(y)_\m m) = r_\<(x \cdot_\l y)_{\l+\m}\phi(m), \\&
	l_\<(x \>_\l y)_{\l+\m}\psi(m) = l_\>(\a(x))_\l (l_\<(y)_\m m), \\&
	r_\<(\b(x))_\l( l_\>(y)_\m m) = l_\>(\a(y))_\m (r_\<(x)_\l m), \\&
	r_\<(\b(x))_\l (r_\>(y)_\m m) = r_\>(y \<_{\m=-\p-\l} x)_{\l+\m} \phi(m), \\&
	l_\>(x \cdot_\l y)_{\l+\m}\psi(m) = l_\>(\a(x))_\l (l_\>(y)_\m m), \\&
	r_\>(\b(x))_\l( l_\cdot(y)_\m m) = l_\>(\a(y))_\m (r_\>(x)_\l m), \\&
	r_\>(\b(x))_\l (r_\cdot(y)_\m m) = r_\>(y \>_\m  x)_{\l+\m}\phi(m),\\&
	\phi(l_\<(x)_\l m) = l_\<(\a(x))_\l \phi(m),  \\&
	\phi(r_\<(x)_\l m) = r_\<(\a(x))_\l\phi(m),  \\&
	\psi(l_\<(x)_\l m) = l_\<(\b(x))_\l \psi(m), \\&
	\psi(r_\<(x)_\l m) = r_\<(\b(x))_\l \psi(m),  \\&
	\phi(l_\>(x)_\l m) = l_\>(\a(x))_\l \phi(m), \\&
	\phi(r_\>(x)_\l m) = r_\>(\a(x))_\l \phi(m),  \\&
	\psi(l_\>(x)_\l m) = l_\>(\b(x))_\l \psi(m), \\&
	\psi(r_\>(x)_\l m) = r_\>(\b(x))_\l \psi(m)
	\end{align} where $ x \cdot_\l y = x \<_\l y + x \>_\l y, l_\cdot = l_\< + l_\>$  and  $r_\cdot = r_\< + r_\>. $
\end{defn}
\begin{prop}
	Let $(l_\<,r_\< ,l_\>,r_\>,  \phi, \psi, M)$ be a conformal-representation of a BiHom-dendriform conformal algebra $(B, \<_\l,\>_\l, \a, \b)$, where $M$ is a $ \mathbb{C}[\p] $- module, $\phi,\psi$ are  $ \mathbb{C} $-linear maps, satisfying $\p\phi=\phi\p$ and $\p\psi=\psi\p$. Then, the direct sum $B\oplus M$ of vector spaces is turned into a BiHom-dendriform conformal algebra by defining $\l$-multiplication $\<'_\l$ and $ \>'_\l $ in $B \oplus M$ as follows
	\begin{equation}
	\begin{aligned}(x_1 +m_1)\<'_\l(x_2 + m_2) &:= x_1 \<_\l x_2 + (l_\<(x_1)_\l{ m_2} + r_\<(x_2)_{-\p-\l} m_1),\\(x_1 +m_1)\>'_\l(x_2 + m_2) &:= x_1 \>_\l x_2 + (l_\>(x_1)_\l{ m_2} + r_\>(x_2)_{-\p-\l} m_1),\\
	(\a \oplus \phi)(x  + m) &:= \a(x ) + \phi(m),\\
	(\b \oplus \psi)(x  + m) &:= \b(x) + \psi(m ),
	\end{aligned}
	\end{equation}for all $x,x_1, x_2 \in B,m, m_1, m_2 \in M$ and $\l\in \mathbb{C}$.
\end{prop} 	We denote this BiHom-dendriform conformal algebra by $(B\oplus M,\<'_\l, \>'_\l, \a+\b, \phi+\psi)$, or simply $(B\ltimes_{l_\<,r_\< ,l_\>,r_\>,\a,\b,\phi,\psi}M)$. 
\begin{prop} Let $(l_\<,r_\< ,l_\>,r_\>, \phi, \psi, M)$ be a conformal bimodule of BiHom-dendriform conformal algebra $(B, \<_\l,\>_\l, \a, \b)$. Let $ (B, \cdot_\l =\<_\l + \>_\l, \a,\b) $ be the associative conformal algebra. Then $(l_\<+ l_\>, r_\< +r_\>,  \phi, \psi, M)$ is a conformal bimodule of $ (B, \cdot_\l =\<_\l + \>_\l, \a,\b) $.
\end{prop}\begin{proof} Let us proof the $10$th identity in Definition \ref{def2.6}, other cases can be proved similarly,
	\begin{align*}&(l_\<+ l_\>)(x \cdot_\l y)_{\l+\m}\psi(m)\\&
	=(l_\<+ l_\>)(x \<_\l y + x \>_\l y)_{\l+\m}\psi(m)\\&
	=l_\<(x \<_\l y)_{\l+\m}\psi(m) + l_\<(x \>_\l y)_{\l+\m}\psi(m) + l_\>(x \cdot_\l  y)_{\l+\m}\psi(m)\\&
	={l_\<(\a(x))}_\l {(l_\<+ l_\>)(y)}_\m m + {l_\>(\a(x))}_\l {{l_\<(y)}_\m m }+ {l_\>(\a(x))}_\l {{l_\>(y)}_\m m}\\&
	=l_\<(\a(x))_\l ((l_\<+ l_\>)(y)_\m m) + l_\>(\a(x))_\l ((l_\<+ l_\>)(y)_\m m)\\&
	=(l_\<+ l_\>)(\a(x))_\l ((l_\<+ l_\>)(y)_\m) m.\end{align*}\end{proof}Now we introduce the notion of a BiHom-pre-Poisson conformal algebra and give some important results.\begin{defn}
	A non-commutative BiHom-pre-Poisson conformal algebra is a $ 6 $-tuple $(B, \<_\l, \>_\l, *_\l,  \a, \b)$
	such that $(B, \<_\l, \>_\l, \a, \b)$ is a BiHom-dendriform algebra and $(B,  *_\l,  \a, \b)$ is a BiHom-preLie conformal algebra satisfying the following compatibility conditions:
	\begin{eqnarray}\label{eq35}
	\begin{aligned}
	(\b(x) *_\l \a(y) - \b(y) *_{-\p-\l} \a(x)) \<_{\l+\m} \b(z) = \a\b(x) *_\l (\a(y) \<_\m z) - \a\b(y) \<_\m (\a(x) *_\l z),\\
	\b(x) \>_\l (\a\b(y) *_\m \a(z) - \b(z) *_{-\p-\m}\a^2(y)) = \a\b^2(y) *_\m (x \>_\l\a(z)) - (\b^2
	(y) *_{\m} x) \>_{\l+\m} \a\b(z),\\
	(\b(x) \<_\l \a(y) + \b(x) \>_\l \a(y)) *_{\l+\m} \b(z) = (\b(x) *_\l\a(z)) \>_{-\p-\m} \b(y) +\a\b(x) \<_\l (\a(y) *_{\m} z)
	\end{aligned}
	\end{eqnarray}
\end{defn}
\begin{thm}
	Let $ (B, \<_\l,\>_\l, *_\l) $ be a pre-Poisson conformal algebra and $\a,\b\in End(B)$ be two commuting morphisms of $B$. Then $B_{\a,\b}:= (B, \<_{\a,\b} =\<\circ (\a\otimes \b),\>_{\a,\b} =\>\otimes
	\circ(\a\otimes \b), *_{\a,\b} =*\circ (\a\otimes \b), \a,\b)$ is a  BiHom-pre-Poisson conformal algebra, known as the Yau twist of $B$. Moreover, assume that there is another such BiHom-pre-Poisson conformal algebra $B'_{\a,\b}:= (B', \<'_{\a,\b} =\<'
	\circ(\a'\otimes \b'),\>'_{\a',\b'} =\>'
	\circ(\a'\otimes \b'), *'_{\a',\b'} =*'\circ (\a'\otimes \b'), \a',\b')$ generated from the pre-Poisson conformal algebra $ B' $ in the presence of structure maps $ \a'\b' $ and multiplication maps $\<,\>, *$. Assume that $f : B \to B'$
	be a pre-Poisson conformal algebra morphism that satisfiy $f\circ \a'=\a \circ f$ ,$f\circ \b'=\b \circ f$ Then $ f : B_{\a,\b} \to B_{\a',\b'}$ is a BiHom-pre-Poisson conformal algebra morphism.
\end{thm}
\begin{proof} We shall only prove first relation in the Eq. (\ref{eq35}), the others being proved analogously. Then, for
	any $ x, y, z\in B$ and $\l,\m\in \mathbb{C}$
	\begin{eqnarray}\begin{aligned}
	&(\b(x)*^{\a,\b}_\l \a(y) - \b(y)*^{\a,\b}_{-\p-\l} \a(x)) \<^{\a,\b}_{\l+\m} \b(z)
	\\&=(\a\b(x) *_\l \a\b(y) - \a\b(y) *_{-\p-\l} \a\b(x)) \<^{\a,\b}_{\l+\m}\b(z)
	\\&=(\a^2\b(x) *_{\l}{\a^2\b}(y) - \a^2\b(y) *_{-\p-\l}\a^2\b(x)) \<_{\l+\m} \b^2(z)
	\\& =\a^2\b(x) *_\l (\a^2\b(y) \<_\m \b^2(z))-\a^2\b(y) \<_\m (\a^2\b(x) *_\l\a^2(z))
	\\& =\a\b(x) *_\l^{\a,\b}(\a(y) \<_\m^{\a,\b}z) - \a\b(y) \<_\m^{\a,\b}(\a(x) *_{\l}^{\a,\b}z)\\&=
	\a\b(x) *_\l^{\a,\b} (\a(y) \<_\m^{\a,\b} z) - \a\b(y) \<_\m^{\a,\b} (\a(x) *_\l^{\a,\b}z).	
	\end{aligned}
	\end{eqnarray}
	For the second assertion, we have
	\begin{eqnarray}
	\begin{aligned}
	f(x \<_\l^{\a,\b}y) &=f(\a(x) \<_\l \b(y))\\&
	=f(\a(x)) \<'_\l f(\b(y))\\&
	=\a' f(x) \<'_\l\b' f(y)\\&
	=f(x) \<_\l^{'\a,\b}f(y).
	\end{aligned}
	\end{eqnarray}Similarly, we have 	$ f(x \<_\l^{\a,\b}y) =f(x) \<^{'\a,\b}_\l f(y). $ and $ f(x*_\l^{\a,\b} y) = f(x)*_\l^{'\a,\b}f(y)$. This completes the proof.\end{proof}
\begin{prop}\label{prop3.13}
	More generally, let $ (B,\<_\l,\>_\l, *_\l, \a, \b) $ be a commutative BiHom-pre-Poisson conformal algebra and $\a'\b': B\to B$ be a two noncommutative BiHom-pre-Poisson algebra morphisms such that any two of the maps $\a,\b,\a',\b'$ commute. Then $(B, \<^{\a,\b}_\l, \>^{\a,\b}_\l, *^{\a,\b}_\l, \a\circ\a', \b\circ b')$ is a noncommutative BiHom-pre-Poisson conformal algebra.
\end{prop}
\begin{cor}
	Let $ (B,\<_\l,\>_\l, *_\l, \a, \b) $ be a noncommutative BiHom-pre-Poisson conformal algebra
	and $m \in N/\{0\}$, Then
	\begin{enumerate}
		\item The $mth$ derived noncommutative BiHom-pre-Poisson algebra of type $ 1 $ of $ B $ is defined
		by $$ B^m_1 = (B, \<_{\l}^{(m)} =\<_{\l}\circ (\a^m \otimes \b^m), \>_{\l}^{(m)}=\>_{\l} \circ(\a^m \otimes \b^m), *_{\l}^{(m)} = *_\l \circ (\a^m \otimes \b^m), \a^{m+1}, \b^{m+1}). $$
		\item The $mth $ derived  BiHom-pre-Poisson conformal algebra of type $ 2 $ of $B$ is defined
		by $$ B^m_2 = (B, \<_\l^{(2^m-1)}=\<_\l\circ (\a^m \otimes \b^m), \>_\l^{(2^m-1)}=\>_\l \circ(\a^{2^m-1}\otimes \b^{2^m-1}), *_\l^{(2^m-1)} = *_\l \circ (\a^m\otimes \b^m), \a^{2m}, \b^{2m}). $$
	\end{enumerate}
\end{cor}
\begin{proof}Apply Proposition \ref{prop3.13} with $\a' = \a^m, \b' = \b^m$ and $\a' = \a^{2m-1}, \b' = \b^{2m-1}$ respectively.
\end{proof}
\begin{thm}
	Let $(B, \<_\l, \>_\l, *_\l,  \a, \b)$ be a regular noncommutative BiHom-pre-Poisson
	conformal algebra. Then  $(B, [._\l.], \cdot_\l,  \a, \b)$ is a noncommutative BiHom-Poisson conformal algebra with
	$x \cdot_\l y = x\<_\l y + x \>_\l y,$  and  $[x_\l y] = x*_\l y- \a^{-1}\b(y) *_{-\p-\l} \a\b^{-1}(x) $, for any $ x, y\in B$, $\l\in \mathbb{C}$ We say that $(B, [._\l.], \cdot_\l,  \a, \b)$ is the sub-adjacent noncommutative BiHom-Poisson conformal algebra of $(B, \<_\l, \>_\l, *_\l,  \a, \b)$ and denoted by $B^c$.
\end{thm}
\begin{proof} By Proposition \ref{prop3.2} and Lemma \ref{lem3.7}, we deduce that $(B, \cdot_\l, \alpha, \beta)$ is a  conformal algebra and  $ (B, [\cdot_\l \cdot] , \alpha, \beta) $   is a BiHom-Lie conformal algebra. \\
	Now, we show the BiHom-Leibniz conformal identity 
	\begin{align*}
	&[\alpha \beta(x)_\l (y \cdot_\m z)] - [\beta(x)_\l y] \cdot_{\l+\m} \beta(z) - \beta(y) \cdot_\m [\alpha(x)_\l z] \\
	=& [\alpha \beta(x)_\l (y \prec_\m z + y \succ_\m z)] -  [\beta(x)_\l y] \prec_{\l+\m} \beta(z) - [\beta(x)_\l y] \succ_{\l+\m }\beta(z) \\
	&- \beta(y) \prec_\m  [\alpha(x)_\l z]  - \beta(y) \succ_\m [\alpha(x)_\l z] \\
=& \alpha \beta(x) *_\l (y \prec_\m z) - \alpha^{-1}\beta(y \prec_\m z) *_{-\p-\l} \a^{2}(x) \\
	&+ \alpha \beta(x) *_\l (y \succ_\m z) - \alpha^{-1}\beta(y \succ_\m z) *_{-\p-\l} \a^{2}(x) \\
	&- (\beta(x) *_\l y) \prec_{\l+\m} \beta(z) + (\alpha^{-1} \beta(y) *_{-\p-\l} \alpha(x))\prec_{\l+\m} \beta(z) \\
	&- (\beta(x) *_{\l} y) \succ_{\l+\m} \beta(z) + (\alpha^{-1} \beta(y) *_{-\p-\l} \alpha(x))\succ_{\l+\m} \beta(z) \\
	&- \beta(y) \prec_\m (\alpha(x) *_\l z) + \beta(y) \prec_{\m} (\alpha^{-1} \beta(z) *_{-\p-\l} \a^2\beta^{-1} (x)) \\
	&- \beta(y) \succ_\m (\alpha(x) *_\l z) + \beta(y) \succ_\m (\alpha^{-1} \beta(z) *_{-\p-\l} \a^2\beta^{-1}(x)) 
	\\=& \{^1\alpha \beta(x) *_\l (y \prec_\m z)- ^5(\beta(x) *_\l y) \prec_{\l+\m} \beta(z) \\&- ^9\beta(y) \prec_\m (\alpha(x) *_\l z) + ^{10} \beta(y) \prec_{\m} (\alpha^{-1} \beta(z) *_{-\p-\l} \a^2\beta^{-1} (x)) \}\\&+\{^2- \alpha^{-1}\beta(y \prec_\m z) *_{-\p-\l} \a^{2}(x)- ^4 \alpha^{-1}\beta(y \succ_\m z) *_{-\p-\l} \a^{2}(x)\\&+ ^6(\alpha^{-1} \beta(y) *_{-\p-\l} \alpha(x))\prec_{\l+\m} \beta(z) + ^8(\alpha^{-1} \beta(y) *_{-\p-\l} \alpha(x))\succ_{\l+\m} \beta(z) \}\\&+\{ ^3+ \alpha \beta(x) *_\l (y \succ_\m z)- ^7 (\beta(x) *_{\l} y) \succ_{\l+\m} \beta(z) \\&- ^{11}\beta(y) \succ_\m (\alpha(x) *_\l z) +^{12} \beta(y) \succ_\m (\alpha^{-1} \beta(z) *_{-\p-\l} \a^2\beta^{-1}(x))\}
	\\=& 0 + 0 + 0 = 0.
	\end{align*}Above result is obtained by using Eq. (\ref{eq35}).
\end{proof}In the following we introduce the notions of conformal bimodule of noncommutative BiHom-pre-Poisson conformal algebras and relevant properties are also given \begin{defn}\label{defprerep}Let $ (B, \<_\l,\>_\l,*_\l,\a,\b) $ be a  BiHom-pre-Poisson conformal algebra. A conformal bimodule of $B$ is a $9$-tuple $(l_\<, r_\>, l_\>, r_\<, l_*, r_*,\phi,\psi, M)$ such that $(l_*, r_*, \phi,\psi, M)$ is a conformal bimodule of the BiHom-pre-Lie confomal algebra $(B, *_\l, \a, \b)$ and $(l_\<, r_\<, l_\>, r_\>, \phi,\psi, M)$ is a conformal bimodule of the BiHom-dendriform conformal algebra $(B, \<_\l,\>_\l,\a,\b)$ satisfying for all $x, y \in B$ and $m \in M$ :
	\begin{align}
	l_\<([\b(x)_\l \a(y)])_{\l+\m}\psi(m) &= l_*(\a\b(x))_\l (l_\<(\a(y))_\m m) - l_\<(\a\b(y))_\m (l_*(\a(x))_\l m),\\
	r_\<(\b(x))_\l(\rho(\b(y))_\m\phi(m)) &= l_*(\a\b(y))_\m (r_\<(x)_\l \phi(m))- r_\<(\a(y) *_\m x)_{\l+\m}\phi\psi(m),\\
	-r_\<(\b(x))_\l(\rho(\b(y))_\m\phi(m))& = r_*(\a(y) \<_\m x)_{\l+\m}\phi \psi(m) - l_\<(\a\b(y))_\m (r_*(x)_\l(\phi(m))),\\
	l_\>(\b(x))_\l (\rho(\a\b(y))_\m\phi(m)) &= l_*(\a\b^2(y))_\m (l_\>(x)_\l \phi(m) )- l_\>(\b^2(y) *_\m x)_{\l+\m}\phi\psi(m),\\
	r_\>([\a\b(x)_\l \a(y)])_{\l+\m}\psi(m) &= l_*(\a\b^2 (x))_\l (r_\>(\a(y))_\m m)- r_\>(\a\b(y))_{\m}(l_*(\b^2(x))_\l m),\\
	-l_\>(\b(x))_\l (\rho(\b(y))_\m \phi^{2}(m)) &= r_*(x \>_\l \a(y))_{\l+\m}\phi\psi^2
	(m) - r_\>(\a\b(y))_\m (r_*(x)_\l \psi^2(m)), \\
	l_*(\b(x)\cdot_\l \a(y))_{\l+\m}\psi(m) &= r_\>(\b(y))_\m l_*(\b(x))_\l \phi(m) + l_\<(\a\b(x))_\l l_*(\a(y))_\m m,\\
	r_*(\b(x))_\l(l_\cdot(\b(y))_\m \phi(m)) &= l_\>(\b(y) *_\m \a(x))_{\l+\m}\phi(m) + l_\<(\a\b(y))_\m( r_*(x)_\l\phi(m)),\\
	r_*(\b(x))_\l r_\cdot(\a(y))_\m\psi(m) &= r_\>(\b(y))_\m r_*(\a(x))_\l\psi(m) + r_\<(\a(y) *_\m x)_{\l+\m}\phi\psi(m), 
	\end{align}where
	$$ x \cdot_\l y = x \<_\l y + x \>_\l y, l_\cdot = l_\< + l_\>, r_\cdot = r_\< + r_\>,	 $$
	$$ [\b(x)_\l \a(y)] = \b(x) *_\l \a(y) - \b(y) *_{-\p-\l} \a(x), $$
	$$(\rho \circ \b)\phi = (l_* \circ \b)\phi - (r_* \circ \a)\psi. $$\end{defn}
\begin{prop}
	Let $(l_\<,r_\< ,l_\>,r_\>, l_*,r_* \phi, \psi, M)$ be a conformal-Bimodule of a BiHom-pre Poisson conformal algebra $(B, \<_\l,\>_\l,*_\l, \a, \b)$, where $M$ is a $ \mathbb{C}[\p] $- module, $\phi,\psi$ are  $ \mathbb{C} $-linear maps, satisfying $\p\phi=\phi\p$ and $\p\psi=\psi\p$. Then, the direct sum $B\oplus M$ is a BiHom-pre-Poisson conformal algebra by defining $\l$-multiplication $\<'_\l$  $ \>'_\l $ and $*'_\l$ in $B \oplus M$ as follows
	\begin{equation}
	\begin{aligned}(x_1 +m_1)\<'_\l(x_2 + m_2) &:= x_1 \<_\l x_2 + (l_\<(x_1)_\l{ m_2} + r_\<(x_2)_{-\p-\l} m_1),\\(x_1 +m_1)\>'_\l(x_2 + m_2) &:= x_1 \>_\l x_2 + (l_\>(x_1)_\l{ m_2} + r_\>(x_2)_{-\p-\l} m_1),\\(x_1 +m_1)*'_\l(x_2 + m_2) &:= x_1 *_\l x_2 + (l_*(x_1)_\l{ m_2} + r_*(x_2)_{-\p-\l} m_1),\\
	(\a \oplus \phi)(x  + m) &:= \a(x ) + \phi(m),\\
	(\b \oplus \psi)(x  + m) &:= \b(x) + \psi(m ),
	\end{aligned}
	\end{equation}for all $x,x_1, x_2 \in B,m, m_1, m_2 \in M$ and $\l,\m\in \mathbb{C}$.
\end{prop}We denote this BiHom-pre-Poisson conformal algebra by $(B\oplus M,\<'_\l, \>'_\l, *'_\l, \a+\b, \phi+\psi)$, or simply $(B\ltimes_{l_\<,r_\< l_*,r_*,l_\>,r_\>,\a,\b,\phi,\psi}M)$. 
\begin{proof}To show $(B\oplus M,\<'_\l, \>'_\l, *'_\l, \a+\b, \phi+\psi)$ is a BiHom-pre-Poisson conformal algebra, we need to show the axioms  in Eq. (\ref{eq35}), for convenience, we only give the proof of first axiom, other can be proved likewise. 
	For any $x,y,z \in B$, $ m_1,m_2,m_3\in M$ and $\l,\m\in \mathbb{C}$, we have
	\begin{align*}
	&((\b+\psi)(x  + m_1) *'_\l (\a+\b)(y + m_2))\<'_{\l+\m} (\b+\psi)(z + m_3)\\&- ((\b+\psi)(y + m_2) *'_{-\p-\l} (\a+\b)(x  + m_1)) \<'_{\l+\m}(\b+\psi)(z + m_3)\\=&((\b(x) *_\l\a(y)) + l_*(\b(x))_\l\phi(m_2) + r_*(\a(y)_{-\p-\l}\psi(m_1))) \<'_{\l+\m} (\b+\psi)(z + m_3)\\&
	- (\b(y)*_{-\p-\l} \a(x ) + l_*(\b(y)_{-\p-\l} \phi(m_1) + r_*(\a(x))_\l\psi(m_2)) \<'_{\l+\m} (\b+\psi)(z + m_3) \\ =&(\b(x) *_\l\a(y))\<'_{\l+\m} (\b(z) + \psi(m_3)) + (l_*(\b(x))_\l\phi(m_2))\<'_{\l+\m} (\b(z) + \psi(m_3)) \\&+ (r_*(\a(y)_{-\p-\l}\psi(m_1)))\<'_{\l+\m} (\b(z) + \psi(m_3))\\& - (\b(y)*_{-\p-\l} \a(x ))\<'_{\l+\m} (\b(z) + \psi(m_3)) + (l_*(\b(y)_{-\p-\l} \phi(m_1))\<'_{\l+\m} (\b(z) + \psi(m_3))\\& +( r_*(\a(x))_\l\psi(m_2))) \<'_{\l+\m} (\b(z) + \psi(m_3)) \\ 
	%=(\b(x) *_\l\a (y)) \<_{\l+\m}\b(z) + l_\<(\b(x) *_\l \a(y))_{\l+\m}\psi(m_3) \\&+ r_\<(\b(z))_{-\p-\l-\m}l_*(\b(y))\phi(m_1) - r_\<(\b(z))r_*(\a(x ))\psi(m_2) \\&- (\b(y) *\a(x )) \<\b(z) -l_\<(\b(y) * \a(x))\psi(m_3)\\& + r_\<(\b(z))l_*(\b(y))\phi(m_1)-r_\<(\b(z))r_*(\a(x ))\psi(m_2)\\&
	=&[\b(x)_\l\a(y)]_{\l+\m}\psi(m_3)  + l_\<([\b(x )_\l\a(y)])_{\l+\m}\psi(m_3)+ r_\<(\b(z))_{-\p-\l-\m}(\rho(\b(y))_{\l}\phi(m_1)) \\&+ r_\<(\b(z))_{-\p-\l-\m}(\rho(\a(x))_{-\p-\l}\psi(m_2)). \end{align*}
	On the other hand,\begin{align*}&(\a\b(x)+\phi\psi(m_1))*'_\l ((\a(y) + \phi(m_2)) \<_\m (z + m_3))\\&- (\a\b(y)+\phi\psi( m_2))  \<'_\m ((\a(x)+\phi( m_1)) *'_\l(z + m_3))\\
	=&(\a\b(x ) + \phi\psi(m_1))*'_\l(\a(y) \<_\m z + l_\<(\a(y))_\m m_3 + r_\<(z)_{-\p-\m}\phi(m_2)
\\&	- (\a\b(y) + \phi\psi(m_2)) \<'_\m(\a(x) *_\l z + l_*(\a(x))_\l (m_3) + r_*(z)_{-\p-\l}\phi (m_1)) \\
=&\a\b(x) *_\l (\a(y) \<_\m z) + l_*(\a\b(x))_\l (l_\<(\a(y))_\m m_3) \\&
+ l_*(\a\b(x ))_\l (r_\<(z)_{-\p-\m}\phi(m_2)) + r_*(\a(y) \<_\m z)_{-\p-\l}\phi\psi(m_1)\\& -\a\b(y) \<_\m (\a(x) *_\l z) - l_\<(\a\b(y))_\m (l_*(\a(x))_\l m_3)\\&- l_\<(\a\b(y))_\m (r_*(z)_{-\p-\l}\phi(m_1)) - r_\<(\a(x) *_\l z)_{-\p-\m}\phi\psi(m_2).\end{align*}
 Using first three equations of the Defintion \ref{defprerep}, Eq. (\ref{eq35}) and conformal sesqui-linearity, proof is clear. However, $1,2,3$ and $4$ term of the second last equality is equated by the pairs $\{1,5\}, \{2,6\},\{3,8\} $  and $  \{4,7\}$ in the last equality, 
 \end{proof}
\begin{ex}
	    Let $(B, \prec_\l, \succ_\l, *_\l, \a, \b)$ be a noncommutative BiHom-pre-Poisson  conformal algebra.  A regular conformal  bimodule of $B$ is defined as the tuple $(L_\prec, R_\prec, L_\succ, R_\succ, L_*, R_*, \a, \b, B)$, where $L_\prec(x)_\l y = x \prec_\l y$, $R_\prec(x)_\l y = y \prec_{-\p-\l} x$, $L_\succ(x)_\l y = x \succ_\l y$, $R_\succ(x)_\l y = y \succ_{-\p-\l} x$, and $L_*(x)_\l y = x *_\l y$, $R_*(x)_\l y = y *_{-\p-\l} x$, for all $x, y \in B$, $\l\in \mathbb{C}$.
\end{ex}
\begin{prop}
    Let  $(B_1, \prec_\l^1, \succ_\l^1, \star_\l^1, \alpha_1, \alpha_2)$ and $(B_2, \prec_\l^2, \succ_\l^2, \star_\l^2, \beta_1, \beta_2)$ be two  noncommutative BiHom-pre-Poisson conformal algebra and $f$ is the morphism between them. We observe that, by using $f$, we can form a  conformal bimodule of $B_1$ represented as $ (l_\prec^1, r_\prec^1, l_\succ^1, r_\succ^1, l_\star^1, r_\star^1,\phi,\psi, B_2)$ and defined by $ l_\prec^1(x)_\l y = f(x) \prec^2_\l y$, $ r_\prec^1(x)_\l y = y \prec^2_{-\p-\l} f(x)$, $ l_\succ^1(x)_\l y = f(x) \succ^2_\l y$, $ r_\succ^1(x)_\l y = y \succ^2_{-\p-\l} f(x)$ and $l_\star^1(x)_\l y = f(x) \star^2_{\l} y$, $r_\star^1(x)_\l y = y \star^2_{-\p-\l} f(x)$ for all $(x, y) \in B_1 \times B_2$ and $\l,\m\in \mathbb{C}$.
\end{prop}
\begin{proof} We need to show the axioms in Definition (\ref{defprerep}). Here we just proof $7th$ axiom. Other axioms can be proved similarly. For any $ x, y \in B_1$, $  z\in B_2$ and $\l,\m\in \mathbb{C}$, we have 
	\begin{align*}
&l_*^1(\b(x) \cdot_\l^1  \a(y))_{\l+\m} \psi(z) \\&= f(\b(x) \cdot^1_\l  \a(y)) *^2 \psi(z)
\\&= (\psi f(x) \cdot^2_\l\phi f(y)) *^ 2_{\l+\m} \psi(z)
\\&= (\psi f(x) \>^2_\l\phi f(y)+ \psi f(x) \<^2_\l\phi f(y)) *^ 2_{\l+\m} \psi(z)
\\&= (\psi f(x) \>^2_\l\phi f(y)-\psi f(y) \<^2_{-\p-\l}\phi f(x)) *^ 2_{\l+\m} \psi(z)
\\&= ([\psi f(x) _\l\phi f(y)]_2) *^ 2_{\l+\m} \psi(z)
\\&= f([\b(x) _\l\a (y)]_2) *^ 2_{\l+\m} \psi(z)
\\&= (\psi f(x) *^2_\l \phi(z)) \succ^2_{\l+\m} \psi f(y) + \phi\psi f(x) \prec^2_\l (\phi f(y) *^2_\m z) \text{ (by (\ref{eq35}))}
\\&= (f(\b(x)) *^2_\l \phi(z)) \succ^2_{\l+\m} f(\b(y)) + f(\a\b(x)) \prec^2_\l (f(\a(y)) *^2_\m z)
\\&=r_{\succ_1} (\b(y))_\m(f(\b(x)) *^2_\l \phi(z)) + l_{\prec_1} (\a\b(x))(f(\a(y)) *^2 z)
\\&=r _{\succ_1} (\b(y))_\m l^*_1(\b(x))_\l\phi(z) + l_{\prec_1} (\a\b(x))_\l l^*_1(\a(y))_\m z.
	\end{align*}
	It completes the proof
	\end{proof}
\section{BiHom-Poisson conformal algebra and $\mathcal{O}$-operators }In this section we introduce the notions of an $\mathcal{O}$-operator of BiHom-Poisson conformal 
algebras and we give some related properties. For this we first recall the notion of $\mathcal{O}$-operator on  conformal algebra and BiHom-Lie conformal algebra as follows.
\begin{defn}
	    Consider we have a  conformal algebra $ (B, *_\l, \a, \b) $ and a conformal bimodule $(l, r, \phi,\psi, M )$ over $B$. An $\mathcal{O}$-operator is a $\mathbb{C}[\p]-$ module homomorphism $T : M\to B$ associated with $(l, r, \phi,\psi, M)$ if it satisfies the following axioms for all $m_1,m_2\in M$ and $\l\in \mathbb{C}$	:
\begin{align*}
	\a T = T \phi , \quad \b T &= T \psi \\
	 T(m_1) \cdot_\l  T(m_2) &= T(l(T(m_1))_\l m_2 + r(T(m_2))_{-\p-\l} m_1)
\end{align*}
\end{defn}
\begin{lem}\label{lem4.2}If we have an $\mathcal{O}$-operator $T : M\to B$ on a  conformal algebra $ (B, *_\l, \a, \b) $, we can establish a BiHom-dendriform  conformal algebra on the conformal bimodule $(l, r, \phi,\psi, M )$ given by:
	\begin{align*}
 m_1 \>_\l m_2 = l(T(m_1))_\l m_2, \quad m_1 \<_\l m_2 = r(T(m_2))_{-\p-\l}m_1 \quad \textit{ for all }  m_1,m_2\in M, \l\in \mathbb{C}.
	\end{align*}
	 \end{lem}Now we review the notion of an $\mathcal{O}$-operator on a BiHom-Lie conformal algebra that is linked with the conformal  representation. Note that, these $\mathcal{O}$-operator are the generalization of Rota-Baxter operator of $ 0 $ weight .
\begin{defn}

Consider a BiHom-Lie conformal algebra denoted by $(B, [._\l.], \a, \b)$, and let the conformal representation of this algebra is denoted by $(\rho,\phi,\psi,M)$. In this context, an $\mathcal{O}$-operator associated with $(\rho, \phi,\psi, M)$ is a $\mathbb{C}[\p]$-module map $T:M\to B$ that incorporate the following conditions for all $m_1,m_2 \in M, \l\in \mathbb{C}$:
\begin{align*}
\a T = T \phi , \b T &= T \psi\\
[T(m_1) _\l  T(m_2)] &= T(\rho(T(m_1))_\l m_2 + \rho(T(\phi^{-1}\psi (m_2)))_{-\p-\l} \phi \psi^{-1} (m_1)). 
\end{align*}
\end{defn}
\begin{lem}\label{lem4.4}
If we have an $\mathcal{O}$-operator $T: M\to B$ on a BiHom-Lie conformal algebra with respect to the conformal representation $(\rho,\phi,\psi,M)$. we can generate a BiHom-preLie conformal algebra by the following conformal multiplication $*_\l:M\otimes M\to M[\l]$ defined by
\begin{align*}
m_1 *_\l m_2 = \rho(T(m_1))_\l m_2, \quad \forall m_1, m_2 \in M, \l\in \mathbb{C}.
\end{align*}
We denote this BiHom-preLie conformal algebra by $(M, *_\l, \a,\b)$
\end{lem} 
\begin{defn}
 A $\mathbb{C}[\p]$-module homomorophism  $ T : M \to B$ is called an $\mathcal{O}$-operator on a BiHom-Poisson conformal algebra $(B,*_\l,  [._\l.], \a, \b)$ with respect to  the conformal representation $ (l,r, \rho,\phi,\psi, M )$, if $T$ is an $\mathcal{O}$-operator on both $ (B, *_\l,\a,\b) $ the  conformal algebra and $ (B, [._\l.], \a, \b) $ the BiHom-Lie conformal algebra .
\end{defn}
\begin{ex}
    A Rota-Baxter operator on a noncommutative BiHom-Poisson conformal algebra $(B, *_\l, [._\l.], \a,\b)$ with respect to the regular representation is defined as an $\mathcal{O}$-operator on $B$.
\end{ex}
\begin{thm}
	Consider a BiHom-Poisson conformal algebra $(B, *_\l, [._\l.], \a,\b)$ an $\mathcal{O}$-operator and $T : M \to B$ on $ B $ with respect to the conformal representation $ (l,r, \rho,\phi,\psi, M )$. Note that $(M, \<_\l,\>_\l,*_\l,\a,\b)$ becomes a BiHom-pre-Poisson conformal algebra by defining the new operations $ \>_\l,\<_\l $ and $ *_\l$ on $ M $ given by \begin{align*}
	 m_1 *_\l m_2 = \rho(T(m_1))_\l m_2,\quad  m_1 \>_\l m_2= r(T(m_2))_{-\p-\l} m_1,\quad  m_1\<_\l m_2 = l(T(m_1))_\l m_2.
	\end{align*} Furthermore, we have that $T(M) = \{T(m); m \in M\} \subset B$ forms a subalgebra of $B$. Additionally, on $T(M)$, there exists an induced BiHom-pre-Poisson conformal algebra structure described by 
\begin{align*}
T(m_1) *_\l T(m_2) = T(m_1 *_\l m_2),\\T(m_1)\<_\l T(m_2) = T(m_1 \<_\l m_2,\\ T(m_1) \>_\l T(m_2) = T(m_1 \>_\l m_2), \end{align*} for all $m_1,m_2 \in M$ and $\l\in \mathbb{C}$ .
\end{thm}
\begin{proof}Both the Lemma \ref{lem4.2} and the Lemma \ref{lem4.4} imply that $(B, \<_\l,\>_\l,\a,\b)$ is a BiHom-dendriform conformal algebra, and $(B, *_\l, \a,\b)$ is a BiHom-preLie conformal algebra. In this proof, we focus on demonstrating the first axiom of Eq. (\ref{eq35}), while noting that the remaining axioms can be proven in a similar manner. Let's consider   $x, y, z \in M$, $\l,\m\in \mathbb{C}$.
	\begin{align*}
	&(\psi(x) *_\l \phi(y) - \psi(y) *_{-\p-\l} \phi(x)) \<_{\l+\m} \psi(z)
	- \phi\psi(x) *_\l (\phi(y) \<_\m z) + \phi\psi(y) \<_\m (\phi(x) *_\l z)\\&
	=(\rho(T(\psi(x))_\l\psi(y) -\rho (T(\psi(y))_{-\p-\l}\phi(x)) \<_{\l+\m} \psi(z)
	\\&-\rho(T(\phi\psi(x)))_\l(\phi(y) \<_\m z) + l(T(\phi\psi(y)))_\m(\phi(x) *_\l z)\\&=l(T(\rho(T(\psi(x)))_\l\phi(y) -\rho(T(\psi(y)))_{-\p-\l}\phi(x)))_{\l+\m}\psi(z)
	\\&-\rho(T(\phi\psi(x)))_\l l(T(\phi(y)))_\m z + l(T(\phi\psi(y)))_\m (\rho(\phi(x))_\l z)
	\\&=l([T(\psi(x))_\l T(\phi(y))])_{\l+\m}\psi(z) -\rho(T(\phi\psi(x)))_\l l(T(\phi(y)))_\m z
+ l(T(\phi\psi(y)))_\m (\rho(\phi(x))_\l z) = 0 . 
	\end{align*} Above expressions are obtained by using Eq. (\ref{eq17}). Therefore, $ (M, \<_\l,\>_\l, *_\l, \a,\b) $ is a BiHom-pre-Poisson conformal algebra.
	The remaining part is fairly intuitive.
\end{proof}
\begin{cor}
 Consider a BiHom-Poisson conformal algebra $ (B, *_\l, [._\l.], \a,\b) $ . In that case there exists a BiHom-pre-Poisson conformal algebra structure on $B$ in such a way that its underlying BiHom-Poisson conformal algebra is exactly a BiHom-Poisson conformal algebra $ (B, *_\l, [._\l.], \a,\b) $ iff there exists an invertible $\mathcal{O}$-operator on $ (B, *_\l, [._\l.], \a,\b).$
\end{cor}
\begin{proof}Suppose there exists a bijective $\mathcal{O}$-operator $T: M \to B$ associated with the conformal representation $(l,r,\rho,\phi,\psi, M)$. Then, for all $x, y \in B$, the compatible BiHom-pre-Poisson conformal algebra structure on $B$ is defined as follows:
\begin{align*}
x \<_\l y &= T(l(x)_\l T^{-1}(y)),\\ x \>_\l y &= T(r(y)_{-\p-\l}T^{-1}(x)),\\ x *_\l y &= T(\rho(x)_\l T^{-1}(y)). 
\end{align*}
	Conversely, if $(B, \<_\l,\>_\l,*_\l, \a,\b)$ is a BiHom-pre-Poisson conformal algebra, and $(B, *_\l, [._\l.], \a,\b)$ is the underlying BiHom-Poisson conformal algebra, then the identity map $id$ is an $\mathcal{O}$-operator on $B$ with respect to the regular conformal representation $(L_\<, R_\>, ad,\a,\b, B)$. \end{proof}
\begin{ex}Consider a noncommutative BiHom-Poisson conformal algebra $(B, *_\l, [._\l.], \a,\b)$  and a Rota-Baxter operator $ R : B \to B$ on it. By defining the new operations $\<_\l,\>_\l, \cdot_{\l}$ on $B$, yield the BiHom-pre-Poisson conformal algebra $ (B, \<_\l,\>_\l, \cdot_{\l}, \a, \b) $ defined by \begin{align*} x\cdot_\l y = [R(x)_\l y],x \<_\l y= R(x)*_\l y,x \>_\l y= x *_\l R(y). \end{align*} In this case $ R $ act as a homomorphism of the sub-adjacent BiHom-Poisson conformal algebra $ (B,\cdot'_\l, [._\l.]',\a,\b) $ and $ (B,\cdot_\l, [._\l.] ,\a,\b) $, given by \begin{align*} [x_\l y]'= x *_\l y- \a^{-1}\b(y) *_{-\p-\l}\a\b^{-1}(x)	\quad and \quad  x\cdot'_\l y = x \<_\l y + x\>_\l y. \end{align*}\end{ex}

\end{document}